\documentclass[12pt]{amsart}
\usepackage{amsmath}
\usepackage{amssymb}

\newtheorem{theorem}{Theorem}

\newtheorem{axiom}{Axiom}

\newtheorem{claim}[theorem]{Claim}

\newtheorem{corollary}{Corollary}

\newtheorem{definition}{Definition}

\newtheorem{lemma}{Lemma}
\newtheorem{notation}{Notation}

\newtheorem{remark}[theorem]{Remark}

\begin{document}

\parskip=.075in

\title[Lexicographic Expected Utility]{Generalized Subjective Lexicographic Expected Utility Representation}
\author{Hugo Cruz Sanchez}
\maketitle

\begin{abstract}
We provide foundations for decisions in face of unlikely events by extending the standard framework of Savage to include preferences indexed by a family of events. We derive a subjective lexicographic expected utility representation which allows for infinitely many lexicographically ordered levels of events and for event-dependent attitudes toward risk. Our model thus provides foundations for models in finance that rely on different attitudes toward risk (e.g. Skiadas \cite{Skiadas}) and for off-equilibrium reasonings in infinite dynamic games, thus extending and generalizing the analysis in Blume, Brandenburger and Dekel \cite{BBD1991}.\end{abstract}

\newpage

\section{Introduction}

One of the lessons we learn from the theory of refinements of Nash equilibrium in game theory is that decisions in face of unlikely events play an important role in determining how a game is to be played. In particular, the analysis of dynamic games relies heavily on off-equilibrium reasonings, that is, in determining what would have happened had players not played what they are supposed to play. We also learn from finance theory that the attitudes toward risk may depend on the kinds of events that the agent faces. For instance, it is conceivable that agents become more risk averse in face of catastrophic, unlikely events. Lexicographic Expected Utility (LEU) is a sensible approach to model decisions in face of very unlikely events, as it presumes a hierarchy of events, ordered by relative unlikeliness, and captures the idea that once the agent is faced with an unlikely event, he goes down to the level of the event in the hierarchy and performs a standard expected utility computation. Thus, a LEU model with infinitely many levels and level-dependent attitudes toward risk seems to be the right model to be used in the theory of infinite dynamic games and of financial theories with varying risk attitudes. It turns out, however, that there is no decision theoretic foundations for such a model available in the literature.

This paper fills up this gap. In particular, in a standard Savage-style framework, we consider a decision maker that is described not only by a preference relation over acts, but also by a family of preference relations over acts.\footnote{For now this family is taken as a primitive; later we will argue that each preference in the family can be inferred as a sort of conditional preference from the given preference relation over acts.} Each preference in this family is indexed by some event in the state space. The idea is that a preference indexed by an event, say $\succsim_A$, where $A$ is the indexing event, represents the preferences of the agent when the agent is informed that the event $A$ has occurred. We then provide a list of axioms that such a system of preferences ought to satisfy and show that decisions that are consistent with the axioms can be represented by a Generalized Subjective Lexicographic Expected Utility (GSLEU) functional. Specifically, for a given state space $S$, a sigma-algebra $\Sigma$ on subsets of $S$ and an outcome space $O$, we consider a preference relation $\succsim$ and a family $(\succsim_A)_{A\in\Sigma}$ of preference relations over the space of acts $f:S\to O$. When these preferences satisfy our axioms, it must be that there exists a (possibly uncountable) family of events $\mathcal{E}\subset\Sigma$ and, for each $E\in\mathcal{E}$, a utility function $u_E:O\to\mathbb{R}$ and a subjective probability measure $P_E$ such that an act $f$ is preferred to an act $g$ under $\succsim$ (in short, $f\succsim g$) if and only if the subjective lexicographic expected utility of $f$ is greater than that of $g$. In symbols, it must be that if \[\int u_E(g) dP_E>\int u_E(f) dP_E\] for some $E\in\mathcal{E}$, then there must exist $E'\in\mathcal{E}$ with $E\subset E'$ such that \[\int u_{E'}(f) dP_{E'}>\int u_{E'}(g) dP_{E'}.\]  Moreover, for each $E\in\mathcal{E}$, $P_E$ is uniquely determined, $u_E$ is unique of to affine transformations, and $u_E$ is ordinally equivalent to $u_{E'}$, for any other $E'\in\mathcal{E}$. This last property allows the utility indices $u_E$ and $u_{E'}$ to represent different attitudes toward risk. Observe that the interpretation of the family $\mathcal{E}$ is that of a hierarchy of events, ordered by relative unlikeliness, in that higher level events are interpreted as infinitely more likely than lower level events.

To contrast with the existing literature, the most relevant contribution is that of Blume, Brandenburger and Dekel \cite{BBD1991}, which, by relaxing the Archimedian axiom in an Anscombe and Aumann \cite{AnsAum1963} framework with finitely many states, provide foundations to subjective expected utility representation with finitely many levels and level-independent risk attitudes. That is, they establish the existence of one utility index $u$ and finitely many subjective probability measures $(P_{\ell})_{\ell=1}^L$ such that an act $f$ is preferred to an act $g$ if and only if the existence of a level $\ell$ such that \[\sum u(g(s))P_{\ell}(s)>\sum u(f(s))P_{\ell}(s)\] implies the existence of a level $\ell'<\ell$ such that \[\sum u(f(s))P_{\ell'}(s)>\sum  u(g(s))P_{\ell'}(s).\] It is apparent that the the representation derived here is better suited to the analysis of the problems in game theory and finance mentioned above, as it allows for infinitely many levels and level-depedent utility indices.

Moving on to the axioms, we begin by making precise the interpretation of $\succsim_A$ as the preference when the agent is informed that the event $A$ occurred. We then proceed to relativize the standard Savage axioms to each of the preferences in the family $(\succsim_A)_{A\in\Sigma}$. We note that, because ``preferences when informed of an event $A$'' are part of our primitives, Savage's  \textit{Sure Thing Principle} has an immediate formulation: if the agent prefers act $f$ to act $g$ when informed of an event $A$ and also when informed of the complement of the event $S\backslash A$, then $f$ should indeed be preferred to $g$. We depart from Savage to allow for lexicographic introspection. What we want to capture is a decision maker who, when informed than an extremely unlikely event has occurred, performs the minimal changes in his/her world views in order to make sense of the unlikely event, and then proceeds as a standard expected utility maximizer. In lexicographic terms, this means that we are after a completely ordered hierarchy of levels events, where the order represents that higher levels are infinitely more likely than lower levels. When an unlikely event occurs, the decision maker goes down to the first level at which the event occurs and uses the level's expected utility terms (utility index and subjective probability.) The crucial feature in our axioms that represents such a decision maker is the postulate of existence of a subfamily of events, $\mathcal{E}\subset \Sigma$, satisfying the following properties. First, it is rich enough to identify \textit{relevant events} for the entire family $\Sigma$, where ``relevance" of an event means that it matters for some indexed preference in the family. Second, it is not richer than what is necessary, in the sense that it avoids redundancies. Third, and more important, the family $\mathcal{E}$ connects the corresponding indexed preferences $(\succsim)_{E\in\mathcal{E}}$ with the non-indexed preference $\succsim$, in a lexicographic fashion.

It is important to note that the existence of the family $\mathcal{E}$, as postulated by our axioms, is far from enough for our representation result. In fact, our Theorem \ref{Theo001} shows that the other axioms already imply the existence of a hierarchy of classes of events ordered by ``relative nullity": events in a class $\alpha$ are of ``comparable likelihood", meaning that neither is null relative to the other, but an event in a class $\alpha$ is ``infinitely less likely" than an event in a higher class $\beta\gg\alpha$, meaning that it is null relative to that event. In fact, such other axioms already imply the existence of a qualitative probability for each class, which is the key ingredient for our generalized lexicographic expected utility representation.\footnote{Also, as relativizations of Savage's axioms, they also imply a SEU representation for each event $A\in\Sigma$.} The role played by the postulated family $\mathcal{E}$ is that it provides the necessary ``top events" for each class $\alpha$. That is, we show that, for any class $\alpha$, there exists an event $E\in\mathcal{E}$ that can be interpreted as the ``local state space" for the class $\alpha$: the expected utility representation for class $\alpha$ is determined by $E$, in that every event $A$ in the class $\alpha$ shares the same utility index $u_E$ and the subjective probability $P_A$ can be computed as the conditional of $P_E$ given $A$.

With the axiom system in place, we are able to establish our GSLEU representation result in Theorem \ref{Theo004}. That is, if a decision maker is represented by a preference $\succsim$ and also by a family of indexed preferences $(\succsim_A)_{A\in\Sigma}$, and this system of preferences satisfies our axioms, then choices can be represented by those that maximize the GSLEU functional. An important issue at this juncture, however, is whether the assumed ``informed" preferences $(\succsim_A)_{A\in\Sigma}$ can be inferred from a given ``uninformed" preference $\succsim$ over acts. 

We provide a positive answer to this question, by means of a notion of conditioning the preference $\succsim$ on an event $A$ that captures what the so far primitive notion $\succsim_A$ is meant to capture. Our notion of conditioning is stronger than Savage's notion because it also requires checking for whether ``perturbations" of an act $f$ are (conditionally) preferred to another act $g$. More precisely, Savage's notion of conditioning says that an act $f$ is preferred to an act $g$ conditional on an event $A$ if the act $fAh$ is preferred to the act $gAh$, for any other act $h$, where the notation ``$fAh$" means the act that is equal to $f$ on $A$ and equal to $h$ on the complement of $A$. On top of that, we add that the act $\tilde{f}Ah$ must be preferred to the act $gAh$ and that the act $fAh$ must be preferred to the act $\tilde{g}Ah$, where $\tilde{f}Ah$ represents a perturbation of the act $fAh$ that is equal to $f$ on most of $A$ but equal to some other constant act on a small part of $A$, and is equal to $h$ on the complement of $A$. We show in Theorem \ref{Theo003} that this stronger notion of conditioning characterizes the informed preferences $\succsim_A$, for each $A\in\Sigma$.

In words, we can in principle tease out of $\succsim$ the ``strong" conditional preferences $\succsim_A$, $A\in\Sigma$, by offering choices of acts and perturbed acts. In decision theoretic jargon, this means that the primitives $\succsim_A$, $A\in\Sigma$, are observable. But Theorem \ref{Theo003} is interesting from other perspectives as well. For instance, a classical question in probability is the issue of probabilities conditional on zero probability events. Our strong conditioning provide insights on computing conditionals on infinitely unlikely events by also requiring that the computation be robust to small perturbations. Relatedly, refinements of equilibrium in game theory often require consideration of perturbations of strategies/payoffs. Our notion of conditioning seems to capture exactly the need for such perturbations. These are questions that we plan to address in future research.

\subsection{Related Literature} 

As advanced above, Blume, Brandenburger and Dekel \cite{BBD1991} provide foundations for a SLEU representation that is special in that it does not allow for infinitely many levels or for level-dependent risk attitudes. These issues are consequences of the Anscombe-Aumann framework adopted, since it assumes a finite state space and only considers acts that are mixture-space valued. The extension to an infinite state space is not a simple matter. LaValle and Fishburn \cite{LaValleFishburn} provide an extension of a finite SLEU representation in a finite state space to such a representation in an infinite state space, while still only allowing for finitely many lexicographic levels. The extension is technical in nature, and does not really extend the SLEU model to an infinite version, as it is done here. Our contribution can, in fact, be viewed as a first step in providing a complete infinite extension of the analysis in Blume, Brandenburger and Dekel \cite{BBD1991}: they do not assume, as we do here, the existence of a primitive lexicographic order; the lexicographic representation is a consequence of their relaxation of the Archimedian axiom. Because in our infinite framework we do not even have a well-ordered space of classes of events to begin with, it is not clear how a non-Archimedian approach would work. This is another question for future research.  A related paper is that of Amarante \cite{Ama2013}, which does allow for infinitely many ``informed" (or conditional) preferences by also considering a family of such preferences as a primitive of the model, but does not provide a lexicographic representation, as the focus of the paper is to show that a family of SEU conditional preferences does not necessarily give rise to an SEU unconditional preference. \medskip

We move now to Section \ref{setting and axioms} where the basic framework and axioms are presented. The results are in Section \ref{results} and the conclusion in Section \ref{conclusion}. The Appendix contains the results not proved in the main text.

\section{Setting and Axioms}\label{setting and axioms}

Following Savage \cite{Sav1954} and Machina and Schmeidler \cite{MachSchm1992} (including the naming of axioms), our setting is as follows:  $S$ is the infinite set of states, which capture the uncertainty present in the decision-making process, $\Sigma \subseteq 2^{S}$ is the $\sigma -$%
algebra of events, $O$ is the set of outcomes, $F$ is the set of $\Sigma -$%
measurable functions from $S$ to $O$ with finite range. This is the set of acts
that represents the possible choices in the decision-making process.  A weak order (complete and transitive binary relation) represents the ranking of the decision-maker over $F$. In addition to this classical setting, $\succsim_{A}$ is a weak preference on $F$ for each event $A$ in $\Sigma $.

\def\sm{\backslash}

\begin{notation}
For an event $A$, and a pair of acts $f$ and $g$, $fAg$ represents the composed act that is equal to $f$ on $A$ and $g$ on $S\sm A$.
\end{notation}

\begin{notation}
\label{StrictandIndifference}For a weak order $\succsim $ on a non empty set 
$X$, and a pair of elements of $X$, $f$ and $g$, 
\begin{equation*}
f\succ g\Leftrightarrow \left( f\succsim g\right) \text{ and }\left( \text{%
not }\left( g\succsim f\right) \right)
\end{equation*}%
and%
\begin{equation*}
f\sim g\Leftrightarrow \left( f\succsim g\right) \text{ and }\left(
g\succsim f\right) .
\end{equation*}
\end{notation}

\begin{definition}[\textbf{agree}]
$\succsim _{A}$ agree with $\succsim _{B}$ iff, for each pair of acts $f$
and $g$,%
\begin{equation*}
f\succsim _{A}g\Leftrightarrow f\succsim _{B}g.
\end{equation*}
\end{definition}

Our first axiom on the primitives makes precise the idea that $\succsim_A$ is indeed a ``preference when informed of event $A$". Such definition meets
Savage's idea of ordering if $A$ were known to obtain (page 22, Savage \cite%
{Sav1954}, \textquotedblleft What technical interpretation can be attached
to the idea that $f$ would be preferred to $g$, if $B$ were known to obtain?
Under any reasonable interpretation, the matter would seem not to depend on
the values $f$ and $g$ assume at states outside of $B$\textquotedblright )
In particular, the preference indexed by the empty event is degenerate
(i.e. each act is weakly preferred to any act). However, axioms $P3\frac{1}{2}$ and 
$P5\frac{1}{2}$ below ensure that no weak preference indexed by a non empty event
is degenerate.

\begin{axiom}[$P1\frac{1}{2}$]
For each pair of acts $f$ and $g$, 
\begin{equation*}
f\succsim _{A}g\Rightarrow fAh\succsim _{A}gAh\text{ for each act }h
\end{equation*}%
and%
\begin{equation*}
fAh\succsim _{A}gAh\text{ for some act }h\Rightarrow f\succsim _{A}g\text{.}
\end{equation*}
\end{axiom}

In Savage \cite{Sav1954}, a \emph{null event} is an event that is irrelevant for the
decision-making process. However, in our framework, the decision-making
process has multiple levels, and an event can be irrelevant at some levels,
but relevant for others. Thus, our concept of null event is relative to each
indexed preference.

\begin{definition}[null event]
For each pair of events $A$ and $B$ such that $B\subseteq A$, $B$ is a null
event at $A$ if and only if  $\succsim _{A\sm B}$ agrees with $\succsim _{A}$.
\end{definition}

Note that the concept of null event at $A$ refers only to events (i.e.
belonging to $\Sigma $) contained in $A$. The empty event is null
at every $A$, even when $A$ is the empty event. 
 
The next axiom is the sure thing principle, and it relates different indexed
preferences. It says that the act $f$ is weakly preferred to the act $g$ for
the preference indexed by an event $A$ if $f$ is weakly preferred to $g$ for
the preferences indexed by each event of a bipartition of $A$, and the order
between $f$ and $g$ is strict for the preference indexed by $A$ if it is
strict for the preference indexed by an event of the bipartition which is
non null at $A$. This is consistent with Savage's first formulation of the sure thing
principle (page 21, Savage \cite{Sav1954}), \textquotedblleft Having suggested what
I shall tentatively call the sure-thing principle, let me give it relatively
formal formal statement thus: If the person would not prefer $f$ to $g$,
either knowing that the event $B$ obtained, or knowing that the event $S\sm B$
obtained, then he does not prefer $f$ to $g$. Moreover (provided he does not
regard $B$ as virtually impossible) if he would definitely prefer $g$ to $f$, knowing that $B$ obtained, and, if he would not prefer $f$ to $g$, knowing that $B$ did not obtain, then he definitely prefers $g$ to $f$.\textquotedblright )

\begin{axiom}[$P2\frac{1}{2}$]
For each pair of events $A$ and $B$ such that $B\subseteq A$, and each pair
of acts $f$ and $g$, 
\begin{equation*}
f\succsim _{B}g\text{ and }f\succsim_{A\sm B}g\Leftrightarrow f\succsim _{A}g
\end{equation*}%
and%
\begin{equation*}
B\text{ is non null at }A\Rightarrow \left( f\succ _{B}g\text{ and }%
f\succsim _{A\sm B}g\Rightarrow f\succ _{A}g\right) .
\end{equation*}
\end{axiom}

The next axiom says that constant acts are equally ordered by the
preferences indexed by non empty events, and it is called eventwise
monotonicity.

\begin{axiom}[$P3\frac{1}{2}$]
For each non empty event $A$, and constant acts $f$ and $g$,%
\begin{equation*}
f\succsim _{A}g\Leftrightarrow f\succsim _{S}g.
\end{equation*}
\end{axiom}

The next axiom says that an event is more likely than a second event, if a
prize resulting from the first event is preferred to the same prize
resulting from the second event, independently of the prize. It is called
weak comparative probability - prize independence.

\begin{axiom}[$P4\frac{1}{2}$]
For each triple of events $A,B$ and $C$ such that $B,C\subseteq A$, and
constant acts $f,f^{\prime },g$ and $g^{\prime }$ such that $f\succsim
_{S}f^{\prime }$ and $g\succsim _{S}g^{\prime }$,%
\begin{equation*}
fBf^{\prime }\succsim _{A}fCf^{\prime }\Rightarrow gBg^{\prime }\succsim
_{A}gCg^{\prime }.
\end{equation*}
\end{axiom}

The next axiom implies nondegeneracy for each preference indexed by a non
empty event, and it is called nondegeneracy.

\begin{axiom}[$P5\frac{1}{2}$]
There is at least two constant acts $f$ and $g$ such that $f\succ _{S}g$.
\end{axiom}

The axioms above are readily seem as translations of the usual Savage axioms to our setting with an additional family of informed preferences. The next two axioms are particular of our framework. The first restricts the allowed families of indexed preferences. It says that any allowed family of indexed preferences must contain an essential subfamily that is sufficient to determine if an event is relevant for the family (i.e. an event whose subtraction matters for some indexed preference
of the family). It is called \emph{separability axiom} (SE) because the subfamily separates events (by symmetric difference on $\Sigma $), which means that each event of $\Sigma $ is used in the comparison between a pair of acts, if preferences indexed by non empty events are non degenerate (which is true by $P5\frac{1}{2}$.)

\begin{axiom}[$SE$]
There exists a subfamily $\succsim _{E}$, with $E\in \mathcal{E\subseteq }\Sigma\backslash\{
\varnothing\} $, of the family of indexed preferences satisfies the
following: 

\begin{equation*}
\forall B\in \Sigma \left( 
\begin{array}{c}
\forall E\in \mathcal{E}\left( B\subseteq E\Rightarrow \succsim _{E}\text{
agrees with }\succsim _{E\sm B}\right) \\ 
\Rightarrow \\ 
\forall A\in \Sigma \left( B\subseteq A\Rightarrow \succsim _{A}\text{ agrees
with }\succsim _{A\sm B}\right)%
\end{array}%
\right) ,
\end{equation*}%
and,%
\begin{equation*}
\forall A\in \Sigma ,\forall E\in \mathcal{E}\left( E\subseteq A\Rightarrow 
\begin{array}{c}
\succsim _{A}\text{ agrees with }\succsim _{A\sm E} \\ 
\text{or} \\ 
\succsim _{A}\text{ agrees with }\succsim _{E}%
\end{array}%
\right) .
\end{equation*}
\end{axiom}

In order to better understand what is implied by Axiom $SE$, we argue that it is a necessary condition for any subjective lexicographic utility representation with at least one event of probability one in each
class of the hierarchy of classes. In fact, if we have a subjective lexicographic utility representation with
at least one event of probability one in each class of the hierarchy of classes, then for any class, we have a preference induced by the subjective expected utility representation of the class. Also, for each event of the class (so, a non null event), we have a preference indexed by the event corresponding to the preference induced by the subjective expected utility representation conditioned on the event. The preference indexed by the event of probability one of the class (by assumption) agrees with the preference induced by the subjective expected utility representation of the class, and this is true for any other event of probability one of the class. Any event of the class containing an event of probability one of the class is an event of probability one of the class, in particular, the union of an event of the
class with an event of probability one of the class. Thus, the class contains a cofinal subset of events of probability one, which is a singleton only when the class contains an event containing every event of the class. A preference indexed by an event $A$ containing an event $E$ of this subset does not agree with the preference indexed by $E$ if and only if $A$ is of a higher class of the hierarchy (in this case, the preference indexed by $A\sm E$ agree with the preference indexed by $A$, because $E$ is irrelevant for higher classes). And finally, an event $B$ irrelevant for an event $ E\supseteq B$ of the subset is irrelevant for the class, so for any event $A\supseteq B$ of the class.

On the other hand, when combined with the other axioms above, SE is a sufficient condition for a subjective lexicographic utility representation with at least one event of probability one in each class of the hierarchy of classes.

The next axiom says that the ordering between acts for the preference $\succsim $ coincides with the ordering between acts resulting from using the lexicographic rule of comparison on preferences indexed by chains of events in $\mathcal{E}$ ordered by the set inclusion $\supseteq $.

\begin{axiom}[$P0\frac{1}{2}$]
For each pair of acts $f$ and $g$,%
\begin{equation*}
f\succsim g\text{ }
\end{equation*}%
\begin{equation*}
\text{iff}
\end{equation*}%
\begin{equation*}
g\succ _{E}f\text{ for some }E\in \mathcal{E}\Rightarrow f\succ _{E^{\prime
}}g\text{ for some }E^{\prime }\in \mathcal{E}\text{ st }E^{\prime
}\supseteq E.
\end{equation*}
\end{axiom}

The lexicographic rule of comparison on a (maybe uncountable) infinite and
order-dense hierarchy implies a transitive strict binary relation, but the
indifference binary relation might violate transitivity. %In the case of a family of acts with finite range, as it is the case in this work, this violation cannot happen, because we need only finite levels of the hierarchy to compare a pair of acts\footnote{For example, $f$ is $o_{A}^{f}$ in $A$ and $o_{S\sm A}^{f}$ in $S\sm A$, and $g$ is $o_{B}^{g}$ in $B$ and $o_{S\sm B}^{g}$ in $S\sm B$, where $o_{\cdot }^{\cdot }$ is an outcome. Without loss, $S\sm A$ and $S\sm B$ are events of the higher class (the class of $S$).}. 
Axiom $P0\frac{1}{2}$ allows us to by-pass this intransitivity issue, by working as a criterion on the set of families of indexed preferences. Similarly as above, note that $P0\frac{1}{2}$ is a necessary condition for any subjective lexicographic utility representation of a weak order in general.

Our final axiom is the usual technical requirement to obtain quantitative probabilities out of qualitative ones. It is called small-event continuity.

\begin{axiom}[$P6\frac{1}{2}$]
For each non empty event $A$, and each triple of acts $f,g$ and $h$ such
that $f\succ _{A}g$ and $h$ constant, there exists a finite $\Sigma -$%
measurable partition of $A$, $A_{i}$, such that, for each $i$, 
\begin{equation*}
f\succ _{A}hA_{i}g
\end{equation*}%
and 
\begin{equation*}
hA_{i}f\succ _{A}g.
\end{equation*}
\end{axiom}

\section{Results}\label{results}

We begin with a result that follows from $P1\frac{1}{2}$ to $P5\frac{1}{2}$,
asserting the existence of a hierarchy of disjoint families of events, and
describing its main properties. The proof is a series of results in the appendix A.

\begin{theorem}
\label{Theo001}There is a partition of $\Sigma $ in classes of events
ordered by an irreflexive, transitive and total order $\gg $. The lowest
class is the class containing the empty event (a singleton called trivial
class), and the highest class is the class containing $S$. An event of a non
trivial class $\alpha $ is non null relatively to other event of the class
(i.e. each of the two events is non null at the union of both events), but
it is null relatively to an event of a higher class $\beta $ (i.e. $\beta
\gg \alpha $). Subevents (superevents) of an event of a non trivial class $%
\alpha $ belong to a class no higher (lower) than $\alpha $. Each class $%
\alpha $ is closed for the union of an event of $\alpha $ and an event of a
class no higher than $\alpha $.
\end{theorem}

From $P1\frac{1}{2}$ to $P6\frac{1}{2}$, we have the following theorem
asserting the existence of a subjective expected utility representation for
the indexed preferences, and describing the relations intraclass and
interclass for these representations. The proof is a series of results in appendices A and B.

\begin{theorem}
\label{Theo002} For a non trivial class $\alpha$, there exists a unique (up
to a positive affine transformations) function $u_{\alpha}:O\to\mathbb{R}$,
and for each event $A$ of the class, there exists a unique finitely
additive, convex-valued probability measure\footnote{%
Indeed, $P_{A}\left( \cdot \right) $ is defined on $\left. \Sigma
\right\vert _{A}$, but it\ can be extended to a probability on $\Sigma $ as $P_{A}\left( A\cap \cdot \right) $. We will keep the simpler notation in the paper.} $P_{A}$ on $\Sigma $, such that, for each pair of acts $f$ and $g$,
\begin{equation*}
f\succsim _{A}g\Leftrightarrow \int u_{\alpha }\left( f\right) dP_{A}\geq
\int u_{\alpha }\left( g\right) dP_{A}.
\end{equation*}%
Moreover, for another event $B\subseteq A$ of the class, $P_{A}\left(
C\right) =P_{B}\left( C\right) P_{A}\left( B\right) $ for an event $%
C\subseteq B$. In addition, different non trivial classes admit different
attitudes toward risk when $O$ is not a mixture space.
\end{theorem}

We are now ready to establish our main representation result. It follows from $P0\frac{1}{2}$ to $P6\frac{1}{2}$, and $SE$, and it asserts the existence of a generalized lexicographic subjective
expected utility representation.

\begin{theorem}[Representation Result]
\label{Theo004}There exists a family of real functions on $O$, $\left\{
u_{E}\right\} _{E\in \mathcal{E}}$, with each $u_{E}$ unique (up to positive
affine transformations), and a unique family of finitely additive,
convex-valued probability measures on $\Sigma $, $\left\{ P_{E}\right\}
_{E\in \mathcal{E}}$, such that, for each pair of acts $f$ and $g$,%
\begin{equation*}
f\succsim g
\end{equation*}%
\begin{equation*}
\text{iff}
\end{equation*}%
\begin{equation*}
\begin{array}{c}
\int u_{E}\left( g\right) dP_{E}>\int u_{E}\left( f\right) dP_{E} \\ 
\text{for some }E\in \mathcal{E}%
\end{array}%
\Rightarrow 
\begin{array}{c}
\int u_{E^{\prime }}\left( f\right) dP_{E^{\prime }}>\int u_{E^{\prime
}}\left( g\right) dP_{E^{\prime }} \\ 
\text{for some }E^{\prime }\in \mathcal{E}\text{ st }E^{\prime }\supseteq E%
\end{array}%
.
\end{equation*}
\end{theorem}

\begin{proof}
First, note that $P3\frac{1}{2}$ and $P5\frac{1}{2}$ implies that
preferences indexed by non empty events are non degenerated, so no non empty
event is null at itself, and, by $SE$, for each non empty event $A$, there
is an event $E$ in $\mathcal{E}$,\ containing\ $A$, and belonging to same
class than $A$. Thus, $\mathcal{E}$\ contains a cofinal subset of each non
trivial class. By $SE$, for each pair of events of $\mathcal{E}\ $in a
class, $E$ and $E^{\prime }$, $\succsim _{E}$ agree with $\succsim
_{E^{\prime }}$, so $E\sm E^{\prime }$, $E^{\prime }\sm E$ and $E\Delta E^{\prime
} $ are null events at some event of $\mathcal{E}\ $in the class, containing 
$E\cup E^{\prime }$ (the class is closed for finite unions, and $\mathcal{E}$%
\ contains a cofinal subset of each non trivial class). By $\left( \ref%
{Theo002}\right) $, each event of $\mathcal{E}$\ in the class has a
subjective expected utility representation, and they are equivalent
representations. More specifically, the Bernoulli indexes of those
representations are equal to the Bernoulli index of the class, and the
beliefs (extended to $\Sigma $) assign the same probability to the events.
As every event in the class is contained in some event in that subset of the
class, the belief (extended to $\Sigma $) of the subjective utility
representation of every event in the class which contains an event in that
subset of the class assign probability one to each event in that subset of
the class. In other words, they are essential top events. And, given the derived properties of $\mathcal{E}$, the representation for $\succsim$ follows from $P0\frac{1}{2}$.
\end{proof}

Each essential ``top events" in $\mathcal{E}$\ corresponds to a class, so the
description of the relations intraclass and interclass in Theorem \ref{Theo002} holds for the representation in Theorem \ref{Theo004}.

We now move to the issue of observability of $\succsim_A$, that is, that $\succsim_A$ can be inferred from choices that respect $\succsim$ over acts that ``strongly" reveal the subjective assessments of the decision maker as to the event $A$ relative to its complement $S\sm A$.

\begin{theorem}[Observability of informed preferences]
\label{Theo003}For each non empty event $A$, and each triple of acts $f$, $g$
and $h$, $f\succ _{A}g$ if, and only if, $fAh\succ gAh$ and, for each constant act $k$,
there exists a finite $\Sigma -$measurable partition of $A$, $(A_{i})_{i=1}^n$, such
that, for each $i=1,...,n$, $kA_{i}\left( fAh\right) \succ gAh$ and $fAh\succ
kA_{i}\left( gAh\right) $.
\end{theorem}

\begin{proof}
Note that $fAh\sim gAh$ implies $fAh\sim _{E}gAh$ for any top event $%
E\supseteq A$ of the class of $A$, and, by the lemma order-preserving (see
appendix Proofs), this implies $fAh\sim _{A}gAh$, which implies $f\sim _{A}g$
(eq. $f\succ _{A}g$ implies $fAh\succ _{A}gAh$, and using the same steps and
lemma, we have $fAh\succ gAh$), but the converse is not necessarily true. We
can have $fAh\succ gAh$ and $fAh\sim _{A}gAh$, if $fAh$ is (strictly)
preferred to $gAh$ for some class $\alpha $ lower than the class of $A$, and
for any class higher than $\alpha $, $gAh$ is not (strictly) preferred to $%
fAh$. In what follows we discuss the effects of perturbations on the case $%
fAh\sim _{A}gAh$.

By $P3\frac{1}{2}$ and $P5\frac{1}{2}$, we can choose a pair of constant
acts $k\succ _{S}k^{\prime }$. Besides, $k$ and $k^{\prime }$ can satisfy
the following: $k$ is preferred (by $P3\frac{1}{2}$,\ we do not need to
specify an indexed preference) to any constant act with outcome belonging to 
$f\left( A\right) \cup g\left( A\right) $ (a finite set), and any constant
act with outcome belonging to $f\left( A\right) \cup g\left( A\right) $ is
preferred to $k^{\prime }$. A finite $\Sigma -$measurable partition of $A$, $%
A_{i}$, defines a finer finite $\Sigma -$measurable partition of $A$, $%
P_{f}=\left\{ A_{i}\cap f^{-1}\left( o\right) :i\text{ and }o\in f\left(
A\right) \right\} $ (for $g$, the discussion that follows is analogous). For 
$P_{f}$, we have $\cup _{i}A_{i}\cap f^{-1}\left( o\right) =A\cap
f^{-1}\left( o\right) $, so, if $A\cap f^{-1}\left( o\right) $ is non null
at $A$, $A_{i}\cap f^{-1}\left( o\right) $ is non null at $A$ for some $i$
(name it $i_{o}$). Thus, if for each $o^{\prime }\in f\left( A\right) $ such
that $A\cap f^{-1}\left( o^{\prime }\right) $ is non null at $A$, $k$ is
preferred to the constant act with outcome $o^{\prime }$, and this ordering
is strict at least once (name it $o$), then\footnote{%
If for each $o^{\prime }\in f\left( A\right) $ such that $A\cap f^{-1}\left(
o^{\prime }\right) $ is non null at $A$, $k$ is preferred to the constant
act with outcome $o^{\prime }$, and this ordering is strict at least once
(name it $o$), then 
\begin{equation*}
k\left( A_{i_{o}}\cap f^{-1}\left( o\right) \right) \left( fAh\right) \succ
_{A_{i_{o}}\cap f^{-1}\left( o\right) }fAh
\end{equation*}%
and 
\begin{equation*}
k\left( A_{i_{o}}\cap f^{-1}\left( o\right) \right) \left( fAh\right) \sim
_{A-\left( A_{i_{o}}\cap f^{-1}\left( o\right) \right) }fAh,
\end{equation*}%
so, by $P2\frac{1}{2}$, 
\begin{equation*}
k\left( A_{i_{o}}\cap f^{-1}\left( o\right) \right) \left( fAh\right) \succ
_{A}fAh.
\end{equation*}%
This one element substitution can be repeated for each $o^{\prime \prime
}\in f\left( A\right) $ such that $A_{i_{o}}\cap f^{-1}\left( o^{\prime
\prime }\right) $ is non empty, and, by $P2\frac{1}{2}$ at each step, after
a finite number of one element substitutions, we obtain 
\begin{equation*}
kA_{i_{o}}\left( fAh\right) \succ _{A}fAh.
\end{equation*}%
} $kA_{i_{o}}\left( fAh\right) \succ _{A}fAh$. And, if for each $o^{\prime
}\in f\left( A\right) $ such that $A\cap f^{-1}\left( o^{\prime }\right) $
is non null at $A$, the constant act with outcome $o^{\prime }$ is preferred
to $k^{\prime }$, and this ordering is strict at least once (name it $o$),
then $fAh\succ _{A}k^{\prime }A_{i_{o}}\left( fAh\right) $. Summing up, for
each $o\in f\left( A\right) $ such that $A\cap f^{-1}\left( o\right) $ is
non null at $A$, either $k$ is strictly preferred to the constant act with
outcome $o$, or the constant act with outcome $o$ is strictly preferred to $%
k^{\prime }$, or both, so, for each $o\in f\left( A\right) $ such that $%
A\cap f^{-1}\left( o\right) $ is non null at $A$, and for each finite $%
\Sigma -$measurable partition of $A$, $A_{i}$, 
\begin{equation*}
kA_{i_{o}}\left( fAh\right) \succ _{A}fAh\text{ or }fAh\succ _{A}k^{\prime
}A_{i_{o}}\left( fAh\right) .
\end{equation*}

In addition (for $f$ and $g$ shifted, the discussion that follows is
analogous), if, for each $o\in f\left( A\right) $ such that $A\cap
f^{-1}\left( o\right) $ is non null at $A$, the constant act with outcome $o$
is indifferent to $k^{\prime }$, and, for each $o^{\prime }\in g\left(
A\right) $ such that $A\cap g^{-1}\left( o^{\prime }\right) $ is non null at 
$A$, $k$ is indifferent to the constant act with outcome $o^{\prime }$, then 
$gAh\sim _{A}kAh\succ _{A}k^{\prime }Ah\sim _{A}fAh$, an absurd . So, if $%
k^{\prime }Ah\sim _{A}fAh$, then for at least one $o^{\prime }\in g\left(
A\right) $ such that $A\cap g^{-1}\left( o^{\prime }\right) $ is non null at 
$A$, $k$ is strictly preferred to the constant act with outcome $o^{\prime }$%
, and $kA_{i_{o^{\prime }}}\left( gAh\right) \succ _{A}gAh$. And, if $%
kAh\sim _{A}fAh$, then for at least one $o^{\prime }\in g\left( A\right) $
such that $A\cap g^{-1}\left( o^{\prime }\right) $ is non null at $A$, the
constant act with outcome $o^{\prime }$ is strictly preferred to $k^{\prime
} $, and $gAh\succ _{A}k^{\prime }A_{i_{o^{\prime }}}\left( gAh\right) $.
Summing up, for each finite $\Sigma -$measurable partition of $A$, $A_{i}$,%
\footnote{$k^{\prime }Ah\sim _{A}fAh$ means that perturbations of $fAh$ are
in the upper set of $fAh$, with some in the strict upper set.} 
\begin{equation*}
k^{\prime }Ah\sim _{A}fAh\Rightarrow \exists i\left( kA_{i}\left( gAh\right)
\succ _{A}gAh\right) ,
\end{equation*}%
\begin{equation*}
kAh\sim _{A}fAh\Rightarrow \exists i\left( gAh\succ _{A}k^{\prime
}A_{i}\left( gAh\right) \right) ,
\end{equation*}%
and,%
\begin{equation*}
kAh\succ _{A}fAh\succ _{A}k^{\prime }Ah\Rightarrow \exists i,j\left(
kA_{i}\left( fAh\right) \succ _{A}fAh\succ _{A}k^{\prime }A_{j}\left(
fAh\right) \right) .
\end{equation*}

Concluding, the set of rules above shows that, for $fAh\sim _{A}gAh$, and
using $k$ and $k^{\prime }$, any finite $\Sigma -$measurable partition of $A$%
, $A_{i}$, generates a set of perturbed versions of $fAh$, and a set of
perturbed versions of $gAh$, which meet at some side (strict upper/lower
set) of $fAh\sim _{A}gAh$.

On the other hand, for $fAh\succ _{A}gAh$, $P6\frac{1}{2}$ ensures that, for
each constant act $k$, there exists a finite $\Sigma -$measurable partition
of $A$, $A_{i}$, such that, for each $i$, $kA_{i}\left( fAh\right) \succ
_{A}gAh$ and $fAh\succ _{A}kA_{i}\left( gAh\right) $.

As, for each strict ordering for the preference indexed by $A$ corresponds a
strict ordering for $\succsim $ (Savage-like) conditioned by $A$, we can
conclude that $fAh\succ gAh$ ($gAh\succ fAh$) with $fAh\sim _{A}gAh$ is not
preserved by perturbations in the sense of $P6\frac{1}{2}$.
\end{proof}

As we advanced in the introduction, our concept of conditioning is more intricate than Savage's conditioning. It involves Savage's conditioning, and also conditioning of perturbations of acts. Using
this concept of conditioning, we can define indexed preferences from $\succsim $. In this approach, the axioms are properties of $\succsim $. In particular, $P0\frac{1}{2}$ says that $\succsim $ has an internal
consistency rule that says that more inclusive events are more decisive for the ordering between acts.

\section{Conclusion}\label{conclusion}

We provided a first step into a fully general foundation to subjective lexicographic expected utility. For the applied literature, we provide foundations for source-dependent risk attitudes (e.g. Skiadas \cite{Skiadas}.) For off-equilibrium reasonings in dynamic games, we provide a general theory supporting standard arguments that invoke ``infinitely less likely events" in dynamic games with infinite horizon. Although we do show that there is no need to use a family of indexed preferences as a primitive of the model, we still need to impose an a priori lexicographic order to obtain our representation. For the future, we plan to dispense with such assumption, and obtain a lexicographic order directly from the implied hierarchy of classes of events that follow from our other axioms. We also plan to explore the implications of our ``perturbed" conditionals to the foundations of refinements of equilibria in game theory.

\newpage

\appendix{}

\section{Proofs}

\begin{lemma}
For each $A\in \Sigma$, $\succsim _{A}$ is a weak preference on $F$
such that, for each acts $f$ and $g$,%
\begin{equation*}
f\succ _{A}g\Rightarrow fAh\succ _{A}gAh\text{ for each act }h
\end{equation*}%
and%
\begin{equation*}
fAh\succ _{A}gAh\text{ for some act }h\Rightarrow f\succ _{A}g.
\end{equation*}

\begin{proof}
Observe that $gAh\succsim _{A}fAh$ for some act $h$ implies $g\succsim _{A}f$%
, thus, if $f\succ _{A}g$ then $fAh\succ _{A}gAh$ for each act $h$. Besides,
observe that $g\succsim _{A}f$ implies that $gAh\succsim _{A}fAh$ for each
act $h$, thus, if $fAh\succ _{A}gAh$ for some act $h$, then $f\succ _{A}g$.
\end{proof}
\end{lemma}

\begin{lemma}
\label{lemma 2}For an event $B$, $f\sim _{A}fBh$ for each pair of acts $f$
and $h$, iff, for each event $C$ disjoint of $B$, $fCh\sim _{A}gCh$ for each
triple of acts $f$, $g$ and $h$.

\begin{proof}
Note that $fCh$ and $gCh$ are equal at $B$.
\end{proof}
\end{lemma}

\begin{lemma}
\label{lemma 3}If $B\subseteq A$ is a null event at $A$, then, for each
event $C$ disjoint of $A\sm B$, $fCh\sim _{A}gCh$ for each triple of acts $f$, $%
g$ and $h$.

\begin{proof}
By $P1\frac{1}{2}$, $f\sim _{A\sm B}f\left( A\sm B\right) h$ for each pair of acts 
$f$ and $h$. Thus, the result follows from $\left( \ref{lemma 2}\right) $
and the definition of null event.
\end{proof}
\end{lemma}

In $\left( \ref{lemma 3}\right) $, if, for each event $C$ disjoint of $A\sm B$, 
$fCh\sim _{A}gCh$ for each triple of acts $f$, $g$ and $h$, then $\left( \ref%
{lemma 2}\right) $ implies that $f\sim _{A}f\left( A\sm B\right) h$ for each
pair of acts $f$ and $h$. However, $\succsim _{A\sm B}$ does not need to agree
with $\succsim _{A}$, so $B$ does not need to be a null event at $A$.
Nevertheless, $P1\frac{1}{2}$ to $P3\frac{1}{2}$, and $P5\frac{1}{2}$,
ensure that $B\subseteq A$ is a null event at $A$.

\begin{lemma}[order preserving \ - new version]
\label{lemma 4}If $B\subseteq A$ is a non null event at $A$, then, for each
pair of acts $f$ and $g$,%
\begin{equation*}
f\succsim _{B}g\text{ }\Leftrightarrow fBh\succsim _{A}gBh\text{ for each
act }h.
\end{equation*}

\begin{proof}
By $P1\frac{1}{2}$, (i) $f\succsim _{B}g$ iff $fBh\succsim _{B}gBh$ for each
act $h$, and (ii) $fBh\sim _{A\sm B}gBh$ for each act $h$.

By the first part of $P2\frac{1}{2}$, if $fBh\succsim _{B}gBh$ for each act $%
h$, then $fBh\succsim _{A}gBh$ for each act $h$. As $B$ is non null at $A,$
by the second part of $P2\frac{1}{2}$, if $fBh\succ _{B}gBh$ for each act $h$%
, then $fBh\succ _{A}gBh$ for each act $h$.
\end{proof}
\end{lemma}

\begin{lemma}[order-preserving - old version]
For each $A,B\in \Sigma$ such that $B\subseteq A$, $B$ is non null at $%
A $, and acts $f$ and $g$, 
\begin{equation*}
f\succsim _{B}g\Leftrightarrow \text{ }fBh\succsim _{A}gBh\text{ for each
act }h.
\end{equation*}

\begin{proof}
By definition, $f\succsim _{B}g$ implies $fBh\succsim _{B}gBh$ for each act $%
h$. By sure-thing consistency, and $h\sim _{A\sm B}h$ for each act $h$, $%
fBh\succsim _{A}gBh$ for each act $h$. Analogously, by lemma above, $f\succ
_{B}g$ implies $fBh\succ _{B}gBh$ for each act $h$. By sure-thing
consistency, $B$ is non null at $A$, and $h\sim _{A\sm B}h$ for each act $h$, $%
fBh\succ _{A}gBh$ for each act $h$.
\end{proof}
\end{lemma}

Next lemma says that, a Savage-like null event with relation to an indexed
preference, and contained in the indexing event, is a null event for the
indexed preference.

\begin{lemma}
\label{lemma 3 converse}If, for an event $B\subseteq A$, $fBh\sim _{A}gBh$,
for each triple of acts $f$, $g$ and $h$, then $B$ is a null event at $A$ (a
partial converse for $\left( \ref{lemma 3}\right) $).

\begin{proof}
If $A$ is the empty event, the result is trivial.

Assume that $A$ is a non empty event. By $P1\frac{1}{2}$, $fBh\sim _{A\sm B}gBh$%
, for each triple of acts $f$, $g$ and $h$. If $B$ is non null at $A$, then,
by the second part of $P2\frac{1}{2}$, $fBh\sim _{B}gBh$, for each triple of
acts $f$, $g$ and $h$. By $P1\frac{1}{2}$, $P3\frac{1}{2}$ and $P5\frac{1}{2}
$, $B$ is the empty event, a contradiction ($B$ is non null at $A$). Thus, $%
B $ is null at $A$.
\end{proof}
\end{lemma}

\begin{theorem}[Nullity]
\label{nullity}For each triple of events $A,B$ and $C$ such that $C\subseteq
B\subseteq A$, 
\begin{equation*}
B\text{ is null at }A\Longrightarrow C\text{ is null at }A,
\end{equation*}%
\begin{equation*}
C\text{ and }B\sm C\text{ are null at }A\Longrightarrow B\text{ is null at }A,
\end{equation*}%
and%
\begin{equation*}
C\text{ is null at }B\Longrightarrow C\text{ is null at }A.
\end{equation*}

\begin{proof}
By $\left( \ref{lemma 3}\right) $ and $\left( \ref{lemma 3 converse}\right) $%
, $C$ is null at $A$ if $B$ is null at $A$.

If, $fCh\sim _{A}gCh$ and $f\left( B\sm C\right) h\sim _{A}g\left( B\sm C\right) h$%
, for each triple of acts $f$, $g$ and $h$, then $fCf\left( B\sm C\right) h\sim
_{A}gCf\left( B\sm C\right) h$ and $f\left( B\sm C\right) gCh\sim _{A}g\left(
B\sm C\right) gCh$, for each triple of acts $f$, $g$ and $h$. By transitivity, $%
fBh\sim _{A}gBh$, for each triple of acts $f$, $g$ and $h$. So, by $\left( %
\ref{lemma 3}\right) $ and $\left( \ref{lemma 3 converse}\right) $, $B$ is
null at $A$ if $C$ and $B\sm C$ are null at $A$.

If $B$ is null at $A$, the third case above is a particular case of the
first case above. So, assuming that $B$ is non null at $A$, if $fCh\sim
_{B}gCh$ for each triple of acts $f$, $g$ and $h$, then, by $\left( \ref%
{lemma 4}\right) $, $fCh\sim _{A}gCh$ for each triple of acts $f$, $g$ and $%
h $. So, by $\left( \ref{lemma 3}\right) $ and $\left( \ref{lemma 3 converse}%
\right) $, $C$ is null at $A$ if $C$ is null at $B$.
\end{proof}
\end{theorem}

\begin{definition}[$\geq _{A}$]
For each triple of events $A,B$ and $C$ such that $B,C\subseteq A$, and
constant acts $f$ and $g$ such that $f\succsim _{S}g$, $B$\ is at least as
probable as $C$\ at $A$, and denoted by $B\geq _{A}C$ (for the definitions
of $>_{A}$ and $=_{A}$ see $\left( \ref{StrictandIndifference}\right) $),
when 
\begin{equation*}
fBg\succsim _{A}fCg.
\end{equation*}
\end{definition}

\begin{lemma}
If $B,C\subseteq D$ are null at $D$, then $B=_{D}C$.

\begin{proof}
By nullity, $B\cup C$ is null at $D$, consequently, $\succsim _{D\sm \left(
B\cup C\right) }$ and $\succsim _{D}$ agree. I.e., prizes at $B$ or $C$ are
negligible.
\end{proof}
\end{lemma}

\begin{lemma}
If $B\subseteq D$ is null at $D$, then $B=_{D}\emptyset $.

\begin{proof}
By definition, $\emptyset $ is null at $D$, thus, by lemma above, $%
B=_{D}\emptyset $.
\end{proof}
\end{lemma}

\begin{lemma}
If $A\subseteq B\subseteq C$ then $B\geq _{C}A$. Besides, $B>_{C}A$ iff $B\sm A$
is not null at $C$.

\begin{proof}
First, observe that for each pair of acts $f$ and $g$, by the definition of $%
\succsim _{C\sm B}$ and $\succsim _{A}$, $fBg\sim _{C\sm B}fAg$ and $fBg\sim
_{A}fAg$.

Given constant acts $f$ and $g$ such that $f\succsim _{S}g$, by eventwise
monotonicity or by the definition of $\succsim _{\emptyset }$, $f\succsim
_{B\sm A}g$. By nondegeneracy, $f$ and $g$ can satisfy $f\succ _{S}g$, thus, if 
$B\sm A$ is non null at $C$,then $f\succ _{B\sm A}g$ and, by nullity, $B$ is non
null at $C$ and $B\sm A$ is non null at $B$. If $B\sm A$ is null at $C$, then $%
\succsim _{\left( C\sm B\right) \sqcup A}$ agree with $\succsim _{C}$,
consequently, $fBg\sim _{C}fAg$; otherwise, by the discussion above and the
sure-thing consistency, $fBg\succ _{C}fAg$.
\end{proof}
\end{lemma}

\begin{lemma}
If $A,B\subseteq C$, $A$ is null at $C$ and $B$ is non null at $C$, then $%
A\cup B=_{C}B$.

\begin{proof}
By nullity, as $B$ is non null at $C$, then $A\cup B$ is non null at $C$.
Now, take $B\subseteq A\cup B\subseteq C$ and use the lemma above.
\end{proof}
\end{lemma}

\begin{lemma}
If $C\subseteq D$ and $C$ is non null at $D$, then $C>_{D}\emptyset $.

\begin{proof}
Given constant acts $f$ and $g$ such that $f\succ _{S}g$ (nondegeneracy), by
eventwise monotonicity, $f\succ _{C}g$, consequently, by order-preserving
lemma, 
\begin{equation*}
fCg\succ _{D}gCg=g=f\emptyset g.
\end{equation*}
\end{proof}
\end{lemma}

\begin{lemma}
If $B,C\subseteq D$, $B$ is null at $D$ and $C$ is non null at $D$, then $%
C>_{D}B$.

\begin{proof}
Both lemmas above imply that, given constant acts $f$ and $g$ such that $%
f\succ _{S}g$ (nondegeneracy), 
\begin{equation*}
fCg\succ _{D}f\emptyset g\sim _{D}fBg.
\end{equation*}
\end{proof}
\end{lemma}

\begin{lemma}
$C\neq \emptyset $ is non null at $C$.

\begin{proof}
By nondegeneracy and eventwise monotonicity, $\succsim _{\emptyset }$ and $%
\succsim _{C}$ do not agree.
\end{proof}
\end{lemma}

\begin{lemma}
If $A\subseteq C\neq \emptyset $ is null at $C$ then $C\sm A$ is non null at $C$%
.

\begin{proof}
By nullity and lemma above.
\end{proof}
\end{lemma}

\begin{lemma}
If $A\subseteq B\subseteq C$, $B$ is non null at $C$ and $A$ is null at $C$,
then $A$ is null at $B$.

\begin{proof}
By contradiction, assume that $A$ is non null at $B$. By nondegeneracy and
eventwise monotonicity, there exist constant acts $f$ and $g$ such that $%
f\succ _{A}g$. By order-preserving lemma, as $A$ is non null at $B$, $%
fAh\succ _{B}gAh$ for each act $h$. Again, by order-preserving lemma, as $B$
is non null at $C$, $fAh\succ _{C}gAh$ for each act $h$. Thus, $A$ is non
null at $C$, an absurd.
\end{proof}
\end{lemma}

If $B$ was null at $C$ in the lemma above, $A$ could be null at $B$ (e.g. $%
A=\emptyset $) or $A$ could be non null at $B$ (e.g. $A=B$).

\begin{definition}[qualitative probability]
A relation $\geq _{A}$ between events is a qualitative probability on $A\neq
\emptyset $, iff, for each triple of events $B$, $C$ and $D$ contained in $A$%
, $D$ disjoint of $B\cup C$, 
\[
\geq _{A}\text{is a weak ordering,}
\]%
\[
B\geq _{A}\emptyset \text{, }A>_{A}\emptyset 
\]%
and%
\[
B\geq _{A}C\Leftrightarrow B\cup D\geq _{A}C\cup D.
\]
\end{definition}

\begin{theorem}
\label{QPT}Given $C\in \Sigma$, $C\neq \emptyset $, the relation $\geq
_{C}$ is a qualitative probability on $\left( C,\left\{ A\in \Sigma%
:A\subseteq C\right\} \right) $.

\begin{proof}
(1). $\geq _{C}$ is a weak preference follows from the fact that $\succsim
_{C}$ is a weak preference.

(2). For each $B\subseteq C$, and constant acts $f$ and $g$ such that $%
f\succsim _{S}g$; by eventwise monotonicity ($B\neq \emptyset $), or by
definition of $\succsim _{\emptyset }$ ($B=\emptyset $), $f\succsim _{B}g$.
Using the definition of $\succsim _{B}$, $fBg\succsim _{B}g$ is true. As $%
fBg\sim _{C\sm B}$ $g$, by sure-thing principle, 
\begin{equation*}
fBg\succsim _{C}g=f\emptyset g,
\end{equation*}%
\qquad i.e., $B\geq _{C}\emptyset $.

(3). By nondegeneracy, there are constant acts $f$ and $g$ such that $f\succ
_{S}g$, so, by eventwise monotonicity, $f\succ _{C}g$. Using the lemma after
the definition of $\succsim _{C}$, $fCg\succ _{C}g=f\emptyset g$ is true.
I.e.$C>_{C}\emptyset $.

(4). For $A,B,D\in \Sigma$ such that $A,B,D\subseteq C$ and $D\cap
A=D\cap B=\emptyset $, and constant acts $f$ and $g$; if $C\sm D$ is null at $C$
then $\succsim _{D}$ agree with $\succsim _{C}$ and, using the definition of 
$\succsim _{D}$, 
\begin{eqnarray*}
fBg &\sim &_{D}fAg \\
f\left( B\cup D\right) g &\sim &_{D}f\left( A\cup D\right) g,
\end{eqnarray*}%
so, $B=_{C}A=_{C}\emptyset $ and $B\cup D=_{C}A\cup D$; but if $C\sm D$ is non
null at $C$, as, by the definition of $\succsim _{C\sm D}$, 
\begin{eqnarray*}
fBg &\sim &_{C\sm D}f\left( B\cup D\right) g \\
fAg &\sim &_{C\sm D}f\left( A\cup D\right) g
\end{eqnarray*}%
and, by sure-thing consistency and $C\sm D$ is non null at $C$, 
\begin{eqnarray*}
fBg &\succsim &_{C\sm D}fAg\Leftrightarrow fBg\succsim _{C}fAg \\
f\left( B\cup D\right) g &\succsim &_{C\sm D}f\left( A\cup D\right)
g\Leftrightarrow f\left( B\cup D\right) g\succsim _{C}f\left( A\cup D\right)
g,
\end{eqnarray*}%
then%
\begin{equation*}
fBg\succsim _{C}fAg\Leftrightarrow f\left( B\cup D\right) g\succsim
_{C}f\left( A\cup D\right) g.
\end{equation*}
\end{proof}
\end{theorem}

The lemma below defines the relation $\geq _{D}$ as the unique qualitative
probability ($C,D\neq \emptyset $, by assumption) at any $C\subseteq D$ non
null at $D$.

\begin{lemma}[weak comparative probability - event independence]
Given $A,B,C,D\in \Sigma$ such that $A,B\subseteq C\subseteq D$, $C$
non null at $D$, 
\begin{equation*}
B\geq _{C}A\Longleftrightarrow B\geq _{D}A.
\end{equation*}

\begin{proof}
Given constant acts $f$ and $g$, by order-preserving lemma and $C$ non null
at $D$, 
\begin{equation*}
fAg\succsim _{C}fBg\Leftrightarrow fAg\succsim _{D}fBg.
\end{equation*}
\end{proof}
\end{lemma}

\begin{definition}[fine]
A qualitative probability $\geq _{A}$ ($A\neq \emptyset $) is fine, iff, for
each event $B$ contained in $A$ such that $B>_{A}\emptyset $, there exits a
finite $\Sigma -$ measurable partition of $A$, $A_{i}$, satisfying $%
B>_{A}A_{i}$ for each $i$.
\end{definition}

\begin{lemma}[fineness]
For each $C\neq \emptyset $, $\geq _{C}$ is fine.

\begin{proof}
Let $B\subseteq C$ be a non null event at $C$. By nondegeneracy, there exist
constant acts $f$ and $g$ such that, $f\succ _{S}g$, and, by eventwise
monotonicity, $f\succ _{B}g$. By order-preserving lemma and $B$ is non null
at $C$, $fBg\succ _{C}g$. Next, by small event continuity, there is a finite
partition of $C$, $\left\{ A_{i}\right\} _{i=1}^{n}\subseteq \Sigma$,
such that $fBg\succ _{C}fA_{i}g$ for each $i=1,...,n$. In other words, $%
B>_{C}A_{i}$ for each $i=1,...,n$ (i.e., $\geq _{C}$ is fine).
\end{proof}
\end{lemma}

\begin{definition}[tight]
A qualitative probability $\geq _{A}$ ($A\neq \emptyset $) is tight, iff, $%
B=_{A}C$, for each pair of events $B$ and $C$ contained in $A$, satisfying: 
\[
B\cup D>_{A}C\text{ and }C\cup E>_{A}B,
\]
for each pair of events $D$ and $E$ contained in $A$, and such that, 
\[
D,E>_{A}\emptyset \text{ and }B\cap D=\emptyset =C\cap E.
\]
\end{definition}

\begin{lemma}[fineness and tightness]
\label{FTT}Given $A$,$B\subseteq C$, if $B>_{C}A$ then there exists a finite
partition of $C$, $\left\{ A_{i}\right\} _{i=1}^{n}\subseteq \Sigma$,
such that $B>_{C}A\cup A_{i}$ for each $i=1,...,n$.

\begin{proof}
$B>_{C}A$ implies $fBg\succ _{C}fAg$ for some constant acts $f$ and $g$ such
that $f\succ _{S}g$ (nondegeneracy), $C\neq \emptyset $ and $B$ non null at $%
C$. By small event continuity, there is a finite partition of $C$, $\left\{
A_{i}\right\} _{i=1}^{n}\subseteq \Sigma$, such that $fBg\succ
_{C}fA_{i}\left( fAg\right) =f\left( A\cup A_{i}\right) g$ for each $%
i=1,...,n$. In other words, $B>_{C}A\cup A_{i}$ for each $i=1,...,n$.
\end{proof}
\end{lemma}

Several conclusions can be obtained from facts above (see theorem 3, page
37, \cite{Sav1954}, and page 195, \cite{Fish1971}). Besides, properties
provided above imply that, for each qualitative probability at each non
empty event, there exists a unique \textbf{finitely additive probability}
(fap) representing it (see theorem 14.2, page 195, and its proof in 198-199, 
\cite{Fish1971}), with the caveat that, for any $C\subseteq D$ non null at $%
D $, the fap at $C$ is the fap at $D$ conditioned at $C$. Furthermore, these
properties imply the existence of a \textbf{SEU representation} for each non
empty event (see next appendix).

From the theorem \ref{QPT} proved above, for each $C\neq \emptyset $, $\geq
_{C}$ is a qualitative probability. Using the lemma \ref{FTT} and theorem 4
(page 38, \cite{Sav1954}), $\geq _{C}$ is fine and tight. By the corollary 1
(page 38, \cite{Sav1954}), the only probability measure that almost agrees
with $\geq _{C}$ (if an event is at least as probable as another event, then the probability of the first event is greater than or equal to the probability of the second event), strictly agrees (an event is at least as probable as another event, iff, the probability of the first event is greater than or equal to the probability of the second event) with it. By the theorem 3 (page 37, \cite%
{Sav1954}), there exists one and only one $P_{C}$ on $\left( C,\left\{ A\in 
\Sigma:A\subseteq C\right\} \right) $\ that almost agrees with $\geq
_{C}$ Thus, there exists one and only one $P_{C}$ on $\left( C,\left\{ A\in 
\Sigma:A\subseteq C\right\} \right) $\ that strictly agrees with $\geq
_{C}$.

In those lemmas and theorems other properties of $P_{C}$ are provided. The
most important is convex-valuedness, that is, for each $B\subseteq C$ and $%
\lambda \in \left[ 0,1\right] $ there exists $A\subseteq B$ such that $%
P_{C}\left( A\right) =\lambda P_{C}\left( B\right) $. This property is used
for defining specific partitions of an event with pre-established
probabilities and outcomes.

Let $\left\{ P_{A}\right\} _{A\in \Sigma\backslash \left\{ \emptyset
\right\} }$ be the family of finitely additive probabilities from the
theorem 14.2, \cite{Fish1971}, where $P_{A}$ represents $\geq _{A}$, and if $%
A$ is non null at $B\supseteq A$, by weak comparative probability - event
independence, and uniqueness, $P_{A}$ is $P_{B}$ conditioned on $A$. Of
course, if $C\supseteq B$ and $A$ is non null at $C$, then $B$ is non null
at $C$, $P_{A}$ is $P_{C}$ conditioned on $A$ and $P_{B}$ is $P_{C}$
conditioned on $B$. As it is known, for $D\subseteq A$, $P_{A}\left(
D\right) =\frac{P_{B}\left( D\right) }{P_{B}\left( A\right) }$, $P_{B}\left(
D\right) =\frac{P_{C}\left( D\right) }{P_{C}\left( B\right) }$ and $%
P_{B}\left( A\right) =\frac{P_{C}\left( A\right) }{P_{C}\left( B\right) }$,
so, $P_{A}\left( D\right) =\frac{\frac{P_{C}\left( D\right) }{P_{C}\left(
B\right) }}{\frac{P_{C}\left( A\right) }{P_{C}\left( B\right) }}=\frac{%
P_{C}\left( D\right) }{P_{C}\left( A\right) }$. In the general case,
qualitative probabilities are not the same necessarily, but for each $%
A\subseteq B\subseteq C$, $B\neq \emptyset $, $P_{C}\left( A\right)
=P_{B}\left( A\right) P_{C}\left( B\right) $.

The notion of relative null events defines an ordering in the events
space with elements boundlessly smaller than other elements. When $A$ is
null and $B$ is non null at $C$, for any finite partitioning $\left\{
B_{k}\right\} _{k=1}^{n}$ of $B$, by nullity, for some $k$, $B_{k}$ is non
null at $C$ . I.e., $A$ is \textquotedblleft infinitely\textquotedblright\
smaller than $B$, violating the Archimedean property (see \cite{Hausner1954}%
). As it is shown below, this non-Archimedean ordering defines an
equivalence relation in the events space such that the quocient space is
a linearly ordered set with the straightforward extension of the
non-Archimedean ordering to this quocient space.

\begin{definition}[$\gg $]
Given events $A$ and $B$, $A\gg B$ iff there exists a event $%
C\supseteq A,B$ such that $A$ is non null at $C$, but $B$ is null at $C$.
\end{definition}

\begin{lemma}
$\gg $ is well defined.

\begin{proof}
Let $A$ and $B$ be events such that $A\gg B$. Suppose, by absurd, there
is a event $D\supseteq A,B$ such that $B$ is non null at $D$. Then, for
some event $C\supseteq A,B$, $A$ is non null at $C$ and $B$ is null at $%
C$, but for some event $D\supseteq A,B$, $B$ is non null at $D$.

By nullity, $A\cup B$ is non null at $C$ because $A\subseteq A\cup B$.
Besides, as $B$ is null at $C$, by lemma above, $B$ is null at $A\cup B$.
Moreover, as $B$ is non null at $D$, by nullity, $B$ is non null at $A\cup B$%
, a contradiction.
\end{proof}
\end{lemma}

By definition, $A\gg \emptyset $ for each event $A\neq \emptyset $.

\begin{lemma}
Given events $A$ and $B$, $A\gg B$ iff $A$ is non null at $A\cup B$, but 
$B$ is null at $A\cup B$.

\begin{proof}
Let $A$ and $B$ be events such that $A\gg B$. Then, for some event $%
C\supseteq A,B$, $A$ is non null at $C$ and $B$ is null at $C$. By nullity, $%
A\cup B$ is non null at $C$ because $A\subseteq A\cup B$. Besides, as $B$ is
null at $C$, by lemma above, $B$ is null at $A\cup B$. Moreover, as $A$ is
non null at $C$, by nullity, $A$ is non null at $A\cup B$.
\end{proof}
\end{lemma}

\begin{lemma}[dominance - \protect\cite{Hausner1954}]
$\gg $ is irreflexive and transitive.

\begin{proof}
$A\gg A$ iff $A$ is non null at $A$, so $A\neq \emptyset $, but $A$ is null
at $A$, so $A=\emptyset $, a contradiction.

$A\gg B\gg C$ iff $A$ is non null at $A\cup B$, but $B$ is null at $A\cup B$%
, and $B$ is non null at $B\cup C$, but $C$ is null at $B\cup C$. Then, $%
A,B\neq \emptyset $; and $A\cup B$ is non null at $A\cup B\cup C$,
otherwise, by nullity, $C\gg B$. Besides, if $B\cup C$ is non null at $A\cup
B\cup C$, by lemma above (contrapositive), $B$ is non null at $A\cup B\cup C$%
, and by nullity, $B$ is non null at $A\cup B$, a contradiction. Thus, $%
A\cup B$ is non null at $A\cup B\cup C$, and $B\cup C$ is null at $A\cup
B\cup C$, implying, by nullity, $A\gg C$.
\end{proof}
\end{lemma}

From the non-Archimedean ordering, an equivalence relation on the events
space can be derived, and from this equivalence relation, a partitioning of
this space in equivalence classes of Archimedean-orderable events, resulting
in a non-Archimedean linear ordering on the quocient space.

\begin{definition}[$\approx $]
Given events $A$ and $B$, $A\approx B$ iff $\lnot \left( A\gg B\right) $
and $\lnot \left( B\gg A\right) $.
\end{definition}

By definition, $A\approx \emptyset $ iff $A=\emptyset $. The class of the
empty set is the trivial equivalence class.

\begin{lemma}
Given non trivial events $A$ and $B$, $A\approx B$ iff $A$ and $B$ are
non null at $A\cup B$.

\begin{proof}
By lemma above and definition of $\approx $, $A$ and $B$ are non null at $%
A\cup B$, or both are null at $A\cup B$. As $A$ and $B$ are non trivial,
both are non null at $A\cup B$.
\end{proof}
\end{lemma}

\begin{lemma}
$\approx $ is an equivalence relation on $\Sigma$.

\begin{proof}
See \cite{Hausner1954}.
\end{proof}
\end{lemma}

\begin{lemma}
On $\Sigma/\approx $, $\gg $ is irreflexive, transitive and total.

\begin{proof}
Trivial.
\end{proof}
\end{lemma}

\begin{lemma}[weak comparative probability - extended event independence]
Given $A,B,C,D\in \Sigma$ such that $A,B\subseteq C,D$, $C\approx D\gg
\emptyset $, 
\begin{equation*}
B\geq _{C}A\Longleftrightarrow B\geq _{C\cup D}A\Longleftrightarrow B\geq
_{D}A.
\end{equation*}

\begin{proof}
By weak comparative probability - event independence, using $A,B\subseteq
C,D\subseteq C\cup D$.
\end{proof}
\end{lemma}

The lemma above defines the relation $\geq _{C\cup D}$ as the unique
qualitative probability at $C$ and $D$ such that $C\approx D\gg \emptyset $.
I.e., there is a unique (in the sense above) qualitative probability on each
non trivial equivalence class in $\Sigma/\approx $.

The discussion above shows that an agent would consider non null, at least,
each event of the equivalence class of a non empty event $A$, if $A$ was
relevant for his/her decision making.

\newpage

\section{Proof Sketch of the Representation Theorem}

\subsection{SEU Representation}

\begin{remark}
I will give the proof for simple acts (simple lotteries in \cite{Sav1954}).
The general case demands Uniform Monotonicity.
\end{remark}

\begin{definition}[$L_{A}^{f}$]
For each $A\in \Sigma\backslash \left\{ \emptyset \right\} $ and simple
act $f$, define a simple lottery on $O$ as $L_{A}^{f}=P_{A}\circ \left.
f\right\vert _{A}^{-1}$.
\end{definition}

In short, it is proved that $f\sim _{A}g$ \ if $L_{A}^{f}=L_{A}^{g}$ for
each couple of simple acts $f$ and $g$. The space of simple lotteries on $O$
is endowed with a weak order $\geqq _{A}$ defined as $L_{A}^{f}\geqq
_{A}L_{A}^{g}$ iff $f\succsim _{A}g$, for each couple of simple acts on $A$, 
$f$ and $g$.  It is proved that this lottery space endowed with $\geqq _{A}$ satisfies independence and Archimedean properties. The theorem 8.2 (page 107, \cite{Fish1971}) provides the representation. The Bernoulli index is the same for every $A\in \Sigma \backslash \left\{ \emptyset \right\} $ if $O$ is a mixture space
satisfying weak forms of the independence and Archimedean properties,
otherwise, each class has a Bernoulli index that preserves ordering for
constant acts only.

\begin{definition}[non redundancy]
Given a simple lottery $\sum_{k=1}^{m}q_{k}1_{o_{k}}$, for some finite $%
m\geq 1$, where $1_{o_{k}}\left( o\right) =\left\{ 
\begin{array}{c}
1,o=o_{k} \\ 
0,o\neq o_{k}%
\end{array}%
\right. $, $\sum_{k=1}^{m}q_{k}=1$, $q_{k}\geq 0$ for $k=1,...,m$, its non
redundant representation is $\sum_{k=1}^{n}p_{k}1_{o_{k}}$, for some finite $%
n\geq 1$, where $n\leq m$, $\sum_{k=1}^{n}p_{k}=1$, $p_{k}>0$ for $k=1,...,n 
$, and $o_{k}=o_{l}\Longrightarrow k=l$.
\end{definition}

\begin{lemma}
\label{FLTA}For each $A\in \Sigma\backslash \left\{ \emptyset \right\} $
and simple lottery $L=\sum_{k=1}^{n}p_{k}1_{o_{k}}$ (non redundancy), there
exist a simple act $f$ such that $L_{A}^{f}=L$.

\begin{proof}
For each $A\in \Sigma\backslash \left\{ \emptyset \right\} $ there is a 
$P_{A}$ such that, for each $B\subseteq A$ and $\lambda \in \left[ 0,1\right]
$, there is an event $C\subseteq B$ satifying $P_{A}\left( C\right) =\lambda
P_{A}\left( B\right) $. Given that, choose $A_{1}\subseteq A$ such that $%
P_{A}\left( A_{1}\right) =p_{1}P_{A}\left( A\right) $, $A_{2}\subseteq
A\sm A_{1}$ such that $P_{A}\left( A_{2}\right) =\frac{p_{2}}{1-p_{1}}%
P_{A}\left( A\sm A_{1}\right) $, $A_{3}\subseteq A\sm \left( A_{1}\cup
A_{2}\right) $ such that $P_{A}\left( A_{3}\right) =\frac{p_{3}}{%
1-p_{1}-p_{2}}P_{A}\left( A\sm \left( A_{1}\cup A_{2}\right) \right) $, and so
forth. Defining $f=o_{1}A_{1}...o_{n}A_{n}h$, where $h$ is an arbitrary
simple act, a simple act satisfying $L_{A}^{f}=L$ is obtained.
\end{proof}
\end{lemma}

\begin{lemma}
\label{BLC}For each $A,B\in \Sigma$, $B\subseteq A$, $B$ non null at $A$%
, simple lotteries $L=\sum_{k=1}^{n}p_{k}1_{o_{k}}$ and $L^{\prime
}=\sum_{k=1}^{n^{\prime }}p_{k}^{\prime }1_{o_{k}^{\prime }}$ (non
redundancy), simple acts $f$ and $g$ such that $L_{B}^{f}=L$ and $%
L_{B}^{g}=L^{\prime }$, and $\lambda \in \left( 0,1\right] $, there is $%
C\subseteq B$ such that $P_{A}\left( C\right) =\lambda P_{A}\left( B\right) $%
, $L_{C}^{f}=L$ and $L_{C}^{g}=L^{\prime }$.

\begin{proof}
Take $C=\bigcup\limits_{\substack{ i=1,...,n  \\ j=1,...,n^{\prime }}}%
C_{ij} $ such that $C_{i,j}\subseteq f^{-1}\left( o_{i}\right) \cap
g^{-1}\left( o_{j}^{\prime }\right) \cap B$ and $P_{A}\left( C_{i,j}\right)
=\lambda P_{A}\left( f^{-1}\left( o_{i}\right) \cap g^{-1}\left(
o_{j}^{\prime }\right) \cap B\right) $. It is straightforward that $%
P_{A}\left( C\right) =\lambda P_{A}\left( B\right) $, $L_{C}^{f}=L$ and $%
L_{C}^{g}=L^{\prime }$.
\end{proof}
\end{lemma}

\begin{lemma}
\label{CBL}For each $A\in \Sigma\backslash \left\{ \emptyset \right\} $
and simple acts $f$ and $g$, $f\sim _{A}g$ \ if $L_{A}^{f}=L_{A}^{g}$.

\begin{proof}
For each $A\in \Sigma\backslash \left\{ \emptyset \right\} $ and simple
act $f$, define a simple lottery as $L_{A}^{f}$ $=P_{A}\circ f^{-1}$, which
can be represented by $\sum_{k=1}^{n}p_{k}1_{o_{k}}$(non redundancy).

\begin{claim}
For each $A\in \Sigma\backslash \left\{ \emptyset \right\} $ and each
couple of simple acts $f$ and $g$, $f\sim _{A}g$ if $%
L_{A}^{f}=L_{A}^{g}=1_{o}$.

\begin{proof}
Observe that $f=oBf$ and $g=oCg$, where $B,C\subseteq A$ and $P_{A}\left(
B\right) =P_{A}\left( C\right) =1$. It is straightforward that $P_{A}\left(
B\cap C\right) =1$, so, $A\sm \left( B\cap C\right) $ is null at $A$,
consequently, $\succsim _{A}$ and $\succsim _{B\cap C}$ agree. By definition
of $\succsim _{B\cap C}$, $f\sim _{B\cap C}g$, thus, for $n=1$ the result
follows.
\end{proof}
\end{claim}

As induction hypothesis, assume that, for each $A\in \Sigma\backslash
\left\{ \emptyset \right\} $, each couple of simple acts $f$ and $g$, and
each $m<n$, $f\sim _{A}g$ \ if $L_{A}^{f}=L_{A}^{g}=%
\sum_{k=1}^{m}p_{k}1_{o_{k}}$.

If $L_{A}^{f}=L_{A}^{g}=\sum_{k=1}^{n}p_{k}1_{o_{k}}$ for two simple acts $f$
and $g$, by the definition of $\succsim _{A}$, there are $%
A_{i},B_{i}\subseteq A$ and $P_{A}\left( B_{i}\right) =P_{A}\left(
C_{i}\right) =p_{i}$ for each $i=1,...,n$, and a simple act $h$, such that 
\begin{eqnarray*}
f &\sim &_{A}o_{1}B_{1}...o_{n}B_{n}h \\
g &\sim &_{A}o_{1}C_{1}...o_{n}C_{n}h.
\end{eqnarray*}

Now, take $D=C_{1}\cap B_{n}$ and $E\subseteq C_{n}\sm B_{n}$ such that $%
P_{A}\left( E\right) =P_{A}\left( D\right) $, and define the simple act 
\begin{equation*}
k=o_{1}Eo_{n}Dg.
\end{equation*}

As $g\sim _{A\sm \left( D\cup E\right) }k$, if $D$ is null at $A$ then, by
nullity, $D\cup E$ is null at $A$, $\succsim _{A}$ agree with $\succsim
_{A\sm \left( D\cup E\right) }$, consequently, $g\sim _{A}k$.

However, if $D$ is non null at $A$ then, there are three possibilities:

(1). $o_{n}\sim _{S}o_{1}$, so, by eventwise monotonicity and the definition
of $\succsim _{E}$ and $\succsim _{D}$, 
\begin{eqnarray*}
o_{n}Eo_{1} &\sim &_{E}o_{n}\sim _{E}o_{1}\sim _{E}o_{n}Do_{1} \\
o_{n}Do_{1} &\sim &_{D}o_{n}\sim _{D}o_{1}\sim _{D}o_{n}Eo_{1}
\end{eqnarray*}
so, by sure-thing consistency, 
\begin{equation*}
o_{n}Eo_{1}\sim _{D\cup E}o_{n}Do_{1},
\end{equation*}%
and, given that $D$ is non null at $A$, by order-preserving lemma%
\begin{equation*}
g=\left( o_{n}Eo_{1}\right) \left( D\cup E\right) g\sim _{A}\left(
o_{n}Do_{1}\right) \left( D\cup E\right) g=k.
\end{equation*}

(2). $o_{n}\succ _{S}o_{1}$, so, 
\begin{equation*}
o_{n}Eo_{1}\sim _{D\cup E}g\succ _{D\cup E}k\sim _{D\cup
E}o_{n}Do_{1}\Rightarrow E>_{D\cup E}D,
\end{equation*}%
and, given that $D$ is non null at $A$, by weak comparative probability -
event independence, 
\begin{equation*}
E>_{A}D,
\end{equation*}%
an absurd. Analogously for $k\succ _{D\cup E}g$, consequently, $k\sim
_{D\cup E}g$, so, given that $D$ is non null at $A$, by order-preserving
lemma%
\begin{equation*}
k=k\left( D\cup E\right) g\sim _{A}g\left( D\cup E\right) g=g.
\end{equation*}

(3). $o_{1}\succ _{S}o_{n}$, it is analogous to (2).

Thus, $g\sim _{A}k$ in each possible case. However, repeating this process
for $D^{\prime }=C_{2}\cap B_{n}$ and $E^{\prime }\subseteq \left(
C_{n}\sm B_{n}\right) \sm E$ such that $P_{A}\left( E^{\prime }\right)
=P_{A}\left( D^{\prime }\right) $, and defining the simple act 
\begin{equation*}
k^{\prime }=o_{2}E^{\prime }o_{n}D^{\prime }k,
\end{equation*}%
it is obtained that $k\sim _{A}k^{\prime }$. So, after finite steps, it is
obtained the simple act%
\begin{equation*}
r\sim _{A}o_{1}D_{1}...o_{n-1}D_{n-1}o_{n}B_{n}h,
\end{equation*}%
which, by construction, is equivalent to $g$ at $A$, and, by induction
hypothesis and the definition of $\succsim _{A\sm B_{n}}$ and $\succsim
_{B_{n}} $, satisfies%
\begin{eqnarray*}
r &\sim &_{A\sm B_{n}}f \\
r &\sim &_{B_{n}}f,
\end{eqnarray*}%
so, by sure-thing consistency, 
\begin{equation*}
r\sim _{A}f\text{.}
\end{equation*}
\end{proof}
\end{lemma}

The lemma above has proved that, for each simple lottery $L$, each $A\in 
\Sigma\backslash \left\{ \emptyset \right\} $ and each couple of simple
acts $f$ and $g$, 
\begin{equation*}
L_{A}^{f}=L_{A}^{g}=L\Longrightarrow f\sim _{A}g.
\end{equation*}

\begin{definition}[$\geqq _{A}$]
For each each $A\in \Sigma\backslash \left\{ \emptyset \right\} $ and
lotteries $L$ and $L^{\prime }$, $L\geqq _{A}L^{\prime }$if there are acts $%
f $ and $g$, such that, $L_{A}^{f}=L$, $L_{A}^{g}=L^{\prime }$ and $%
f\succsim _{A}g$. ($\equiv _{A}$ and $\gneqq _{A}$ are defined as usual)
\end{definition}

As it is shown above, each simple lottery is generated by a simple act.
Besides, it is shown above that two simple acts that generate the same
lottery are equivalent. Thus, if a couple of simple acts $f$ and $g$ satisfy 
$f\succsim _{A}g$, $L_{A}^{f}=L$ and $L_{A}^{g}=L^{\prime }$, then each
couple of simple acts $h$ and $k$ satisfying $L_{A}^{h}=L$ and $%
L_{A}^{k}=L^{\prime }$ must satisfy $h\succsim _{A}k$. I.e. $\geqq _{A}$ is
well defined. Moreover, $\geqq _{A}$ is a weak preference.

\begin{lemma}[sure-thing consistency for lotteries]
For each $A,B\in \Sigma$ such that $B\subseteq A$, and lotteries $L_{1}$
and $L_{2}$,%
\begin{equation*}
L_{1}\geqq _{B}L_{2}\text{ and }L_{1}\geqq _{A\sm B}L_{2}\Longrightarrow
L_{1}\geqq _{A}L_{2}
\end{equation*}%
and%
\begin{equation*}
B\text{ is non null at }A\Longrightarrow \left( L_{1}\gneqq _{B}L_{2}\text{
and }L_{1}\geqq _{A\sm B}L_{2}\Longrightarrow L_{1}\gneqq _{A}L_{2}\right) .
\end{equation*}

\begin{proof}
By lemmas \ref{FLTA}, \ref{BLC}, \ref{CBL} and sure-thing consistency, the
result follows.
\end{proof}
\end{lemma}

\begin{lemma}[order-preserving for lotteries]
For each $A,B\in \Sigma$ such that $B\subseteq A$, $B$ is non null at $%
A $, and lotteries $L$ and $L^{\prime }$, 
\begin{equation*}
L\geqq _{B}L^{\prime }
\end{equation*}
\begin{equation*}
\Leftrightarrow
\end{equation*}
\begin{equation*}
 P_{A}\left( B\right) L+\left(
1-P_{A}\left( B\right) \right) L^{\prime \prime }\geqq _{A}P_{A}\left(
B\right) L^{\prime }+\left( 1-P_{A}\left( B\right) \right) L^{\prime \prime }%
\text{ for each }L^{\prime \prime }\text{.}
\end{equation*}

\begin{proof}
Take acts $f$ and $g$ such that $L_{B}^{f}=L$ and $L_{B}^{g}=L^{\prime }$.
Observe that, for each act $h$, $L_{A}^{fBh}=P_{A}\left( B\right) L+\left(
1-P_{A}\left( B\right) \right) L_{A\sm B}^{h}$ and $L_{A}^{gBh}=P_{A}\left(
B\right) L^{\prime }+\left( 1-P_{A}\left( B\right) \right) L_{A\sm B}^{h}$.
Now, by the order-preserving lemma, and using lemmas \ref{FLTA} and \ref{CBL}%
, the result follows.
\end{proof}
\end{lemma}

Observe that, by order-preserving for lotteries, for each $A\in \Sigma$%
, if, for some $B\subseteq A$ such that $B$ non null at $A$, $L\geqq
_{B}L^{\prime }$, then, for each $C\subseteq A$ such that $C=_{A}B$, $L\geqq
_{C}L^{\prime }$. I.e., the ordering between lotteries at some non null
event at $A$ depend of the probability of this event at $A$ and nothing else.

\begin{lemma}
\label{TUL}For each $A,B,C\in \Sigma$ such that $B,C\subseteq A$, $%
B=_{A}C>_{A}\emptyset $, and lotteries $L_{1}$ and $L_{2}$, $L_{1}\geqq
_{B}L_{2}\Leftrightarrow L_{1}\geqq _{C}L_{2}$.

\begin{proof}
By order-preserving for lotteries lemma, 
\begin{equation*}
L_{1}\geqq _{B}L_{2}
\end{equation*}%
\begin{equation*}
\Leftrightarrow
\end{equation*}%
\begin{equation*}
P_{A}\left( B\right) L_{1}+\left( 1-P_{A}\left( B\right) \right) L_{3}\geqq
_{A}P_{A}\left( B\right) L_{2}+\left( 1-P_{A}\left( B\right) \right) L_{3}%
\text{ for each }L_{3}
\end{equation*}%
\begin{equation*}
\Leftrightarrow
\end{equation*}%
\begin{equation*}
P_{A}\left( C\right) L_{1}+\left( 1-P_{A}\left( C\right) \right) L_{3}\geqq
_{A}P_{A}\left( C\right) L_{2}+\left( 1-P_{A}\left( C\right) \right) L_{3}%
\text{ for each }L_{3}
\end{equation*}%
\begin{equation*}
\Leftrightarrow
\end{equation*}%
\begin{equation*}
L_{1}\geqq _{C}L_{2}.
\end{equation*}
\end{proof}
\end{lemma}

\begin{lemma}[small event continuity for lotteries]
For each each $A\in \Sigma\backslash \left\{ \emptyset \right\} $, and
lotteries $L_{1}$,$L_{2}$ and $L_{3}$ such that $L_{1}\gneqq _{A}L_{2}$ and $%
L_{3}$ constant, there exists a finite partition of $A$, $\left\{
A_{k}\right\} _{k=1}^{n}\subseteq \Sigma$, such that 
\begin{equation*}
L_{1}\gneqq _{A}\left( 1-P_{A}\left( A_{k}\right) \right) L_{2}+P_{A}\left(
A_{k}\right) L_{3}
\end{equation*}%
and 
\begin{equation*}
\left( 1-P_{A}\left( A_{k}\right) \right) L_{1}+P_{A}\left( A_{k}\right)
L_{3}\gneqq _{A}L_{2},
\end{equation*}%
for each $k=1,...,n$.

\begin{proof}
By lemma \ref{FLTA} and small event continuity, the result follows.
\end{proof}
\end{lemma}

Observe that versions of sure-thing consistency and small event continuity
for lotteries are simpler than the original versions.

\begin{lemma}[Independence]
For each each $A\in \Sigma\backslash \left\{ \emptyset \right\} $, each 
$\rho \in \left( 0,1\right] $ and each pair of lotteries $L_{1}$ and $L_{2}$%
, 
\begin{equation*}
L_{1}\geqq _{A}L_{2}\Leftrightarrow \rho L_{1}+\left( 1-\rho \right)
L_{3}\geqq _{A}\rho L_{2}+\left( 1-\rho \right) L_{3}\text{ for each }L_{3}.
\end{equation*}

\begin{proof}
First, observe that the lemma can be rewritten as: for each $A,B\in \mathcal{%
S}$ such that $B\subseteq A$ and $B$ is non null at $A$, $\geqq _{A}$ and $%
\geqq _{B}$ agree.

Given $A,B\in \Sigma$ such that $B\subseteq A$ and $B$ is non null at $%
A $, and lotteries $L_{1}$ and $L_{2}$, by lemma \ref{TUL}, whether $%
L_{1}\geqq _{A}L_{2}\Longleftrightarrow L_{1}\geqq _{B}L_{2}$ is true is
independent of $B$, although it is dependent on $P_{A}\left( B\right) $.
Thus, henceforth, the notation $B\left( \alpha \right) $\ is used for an
incognito event in $A$\ with $P_{A}\left( B\left( \alpha \right) \right)
=\alpha $\ .

Given the discussion above, it shall be proved that, for each $\alpha >0$, $%
L_{1}$ and $L_{2}$, $L_{1}\geqq _{A}L_{2}\Longleftrightarrow L_{1}\geqq
_{B\left( \alpha \right) }L_{2}$.

If, for each $\alpha >0$, $L_{1}\equiv _{B\left( \alpha \right) }L_{2}$,
there is nothing to prove, because $\succsim _{A}$ and $\succsim _{B\left(
1\right) }$ agree, consequently, for each acts $f$ and $g$ such that $%
L_{A}^{f}=L_{1}$ and $L_{A}^{g}=L_{2}$, it follows that $L_{B\left( 1\right)
}^{f}=L_{1}$, $L_{B\left( 1\right) }^{g}=L_{2}$ and $f\sim _{B\left(
1\right) }g$, so, $f\sim _{A}g$ and $L_{1}\equiv _{A}L_{2}$. Thus,
henceforth, it is assumed that there exists $\alpha _{0}>0$ such that $%
L_{1}\lneqq _{B\left( \alpha _{0}\right) }L_{2}$. ($\gneqq $ is analogous)

If $\alpha +\beta \leq 1$, and $B\left( \alpha \right) $ and $B\left( \beta
\right) $ are such that $L_{1}\geqq _{B\left( \alpha \right) }L_{2}$ and $%
L_{1}\geqq _{B\left( \beta \right) }L_{2}$ then, by lemma \ref{TUL}, they
can be taken disjoint and, by lemma \ref{FLTA}, there are acts $f$ and $g$
such that $L_{B\left( \alpha \right) }^{f}=L_{B\left( \beta \right)
}^{f}=L_{1}$ and $L_{B\left( \alpha \right) }^{g}=L_{B\left( \beta \right)
}^{g}=L_{2}$, which imply that $f\succsim _{B\left( \alpha \right) }g$ and $%
f\succsim _{B\left( \beta \right) }g$, and by sure-thing consistency, $%
f\succsim _{B\left( \alpha \right) \sqcup B\left( \beta \right) }g$ . As, by
construction, $L_{B\left( \alpha \right) \sqcup B\left( \beta \right)
}^{f}=L_{1}$, $L_{B\left( \alpha \right) \sqcup B\left( \beta \right)
}^{g}=L_{2}$ and $B\left( \alpha \right) \sqcup B\left( \beta \right) $ is
an event $B\left( \alpha +\beta \right) $, $L_{1}\geqq _{B\left( \alpha
+\beta \right) }L_{2}$.

If $B\left( \alpha \right) $ is such that $L_{1}\geqq _{B\left( \alpha
\right) }L_{2}$ then, taken a $n$-fold uniform partition of $B\left( \alpha
\right) $, $\left\{ B_{j}\left( \frac{\alpha }{n}\right) \right\} _{j=1}^{n}$%
, by lemma \ref{FLTA}, there are acts $f$ and $g$ such that $L_{B_{j}\left( 
\frac{\alpha }{n}\right) }^{f}=L_{1}$ and $L_{B_{j}\left( \frac{\alpha }{n}%
\right) }^{g}=L_{2}$ for each $j=1,...,n$, consequently, $L_{B\left( \alpha
\right) }^{f}=L_{1}$, $L_{B\left( \alpha \right) }^{g}=L_{2}$, $f\succsim
_{B\left( \alpha \right) }g$ and, by sure-thing consistency, $f\succsim
_{B_{j}\left( \frac{\alpha }{n}\right) }g$ \ for some $j$, i.e.$L_{1}\geqq
_{B\left( \frac{\alpha }{n}\right) }L_{2}$.

Both paragraphs above and $L_{1}\lneqq _{B\left( \alpha _{0}\right) }L_{2}$
imply that, 
\begin{equation*}
L_{1}\lneqq _{B\left( q\alpha _{0}\right) }L_{2}\text{ for each }q\in 
%TCIMACRO{\U{211a} }%
%BeginExpansion
\mathbb{Q}
%EndExpansion
\cap \left( 0,\frac{1}{\alpha _{0}}\right] .
\end{equation*}

As $L_{1}$ and $L_{2}$ are simple lotteries and $L_{1}\lneqq _{B\left(
\alpha _{0}\right) }L_{2}$, taking the most preferred outcome in $\left\{
w:L_{1}\left( w\right) +L_{2}\left( w\right) >0\right\} $, $v$, by lemma \ref%
{FLTA}, there are acts $f$ and $g$ such that $L_{B\left( \alpha _{0}\right)
}^{f}=L_{1}$ and $L_{B\left( \alpha _{0}\right) }^{g}=L_{2}$, and by small
event continuity, given the constant act $v$, there is a finite partition $%
\left\{ B_{j}\left( \alpha _{0j}\right) \right\} _{j=1}^{n}$ of $B\left(
\alpha _{0}\right) $ such that%
\begin{equation*}
vB_{j}\left( \alpha _{0j}\right) f\prec _{B\left( \alpha _{0}\right) }g
\end{equation*}%
for each $j=1,...,n$.

Observe that, by definition of $v$, using sure-thing consistency if it is
needed, $v\succsim _{B_{j}\left( \alpha _{0j}\right) }g$ for each $j=1,...,n$%
, and, by sure-thing consistency again, $v\succsim _{B\left( \alpha
_{0}\right) }g$. Besides, for each $j$, $\alpha _{0j}<\alpha _{0}$,
otherwise, $P_{B\left( \alpha _{0}\right) }\left( B_{j}\left( \alpha
_{0j}\right) \right) =1$, consequently, $\succsim _{B\left( \alpha
_{0}\right) }$ and $\succsim _{B_{j}\left( \alpha _{0j}\right) }$ agree, and 
$v\prec _{B_{j}\left( \alpha _{0j}\right) }g$, an absurd. Moreover, for some 
$j$, $\alpha _{0j}>0$, otherwise, $\alpha _{0}=0$, an absurd. Thus, given a $%
j$ such that $0<\alpha _{0j}<\alpha _{0}$, given the definition of $v$, if $%
f\succsim _{B\left( \alpha _{0}\right) \sm B_{j}\left( \alpha _{0j}\right) }g$
then $vB_{j}\left( \alpha _{0j}\right) f\succsim _{B\left( \alpha
_{0}\right) \sm B_{j}\left( \alpha _{0j}\right) }g$ . As $vB_{j}\left( \alpha
_{0j}\right) f\succsim _{B_{j}\left( \alpha _{0j}\right) }g$, by sure-thing
consistency, $vB_{j}\left( \alpha _{0j}\right) f\succsim _{B\left( \alpha
_{0}\right) }g$, an absurd, i.e. 
\begin{equation*}
f\prec _{B\left( \alpha _{0}\right) \sm B_{j}\left( \alpha _{0j}\right) }g.
\end{equation*}

Now, take any $\beta \in \left( 0,\alpha _{0j}\right) $. So, for each $%
B\left( \beta \right) \subseteq B_{j}\left( \alpha _{0j}\right) $, given the
definition of $v$, using sure-thing consistency if it is needed, $v\succsim
_{B\left( \beta \right) }g$ and $v\succsim _{B_{j}\left( \alpha _{0j}\right)
\sm B\left( \beta \right) }f$ , so, $vB\left( \beta \right) f\succsim _{B\left(
\beta \right) }g$ and $vB_{j}\left( \alpha _{0j}\right) f\succsim
_{B_{j}\left( \alpha _{0j}\right) \sm B\left( \beta \right) }vB\left( \beta
\right) f$, consequently, as $vB_{j}\left( \alpha _{0j}\right) f\sim
_{B\left( \alpha _{0}\right) \sm B_{j}\left( \alpha _{0j}\right) }vB\left(
\beta \right) f$, by sure-thing consistency, $vB_{j}\left( \alpha
_{0j}\right) f\succsim _{B\left( \alpha _{0}\right) \sm B\left( \beta \right)
}vB\left( \beta \right) f$. Now, if $vB_{j}\left( \alpha _{0j}\right)
f$ $\succsim _{B\left( \alpha _{0}\right) \sm B\left( \beta \right) }g$ then, as $%
vB_{j}\left( \alpha _{0j}\right) f\succsim _{B\left( \beta \right) }g$, by
sure-thing consistency, $vB_{j}\left( \alpha _{0j}\right) f\succsim
_{B\left( \alpha _{0}\right) }g$, an absurd, consequently,%
\begin{equation*}
vB\left( \beta \right) f\prec _{B\left( \alpha _{0}\right) \sm B\left( \beta
\right) }vB_{j}\left( \alpha _{0j}\right) f\prec _{B\left( \alpha
_{0}\right) \sm B\left( \beta \right) }g,
\end{equation*}%
i.e.%
\begin{equation*}
f\prec _{B\left( \alpha _{0}\right) \sm B\left( \beta \right) }g.
\end{equation*}

From what was shown above, if $0<\beta <\alpha _{0j}<\alpha _{0}$ and $%
L_{1}\lneqq _{B\left( \alpha _{0}\right) }L_{2}$, then, by lemma \ref{FLTA}
and lemma \ref{BLC}, 
\begin{equation*}
L_{1}\lneqq _{B\left( \alpha _{0}-\beta \right) }L_{2}
\end{equation*}

The steps above imply that, for each $\beta \in \left( 0,\alpha _{0j}\right) 
$, 
\begin{equation*}
L_{1}\lneqq _{B\left( q\left( \alpha _{0}-\beta \right) \right) }L_{2}\text{
for each }q\in 
%TCIMACRO{\U{211a} }%
%BeginExpansion
\mathbb{Q}
%EndExpansion
\cap \left( 0,\frac{1}{\alpha _{0}-\beta }\right] ,
\end{equation*}%
consequently,%
\begin{equation*}
L_{1}\lneqq _{B\left( x\right) }L_{2}\text{ for each }x\in \left( 0,1\right]
,
\end{equation*}%
implying that $L_{1}\lneqq _{A}L_{2}$.
\end{proof}
\end{lemma}

\begin{corollary}
\label{ICL}For each $A,B\in \Sigma\backslash \left\{ \emptyset \right\} 
$ such that $B\approx A$, $\geqq _{A}$ and $\geqq _{B}$ agree.

\begin{proof}
Take $A\cup B$ and apply the independence lemma.
\end{proof}
\end{corollary}

\begin{lemma}
\label{DBL}For each $A\in \Sigma\backslash \left\{ \emptyset \right\} $%
, if $L_{2}\gneqq _{A}L_{1}$.and $0\leq \rho <\sigma \leq 1$, then $\rho
L_{1}+\left( 1-\rho \right) L_{2}\gneqq _{A}\sigma L_{1}+\left( 1-\sigma
\right) L_{2}$.

\begin{proof}
It follows straightforwardly from the independence lemma (see page 72, \cite%
{Sav1954}). It is just algebra.
\end{proof}
\end{lemma}

\begin{lemma}[Archimedean]
For each $A\in \Sigma\backslash \left\{ \emptyset \right\} $, if $%
L_{2}\gneqq _{A}L_{1}$ and $L_{2}\geqq _{A}L_{3}\geqq _{A}L_{1}$ for some
fixed lottery $L_{3}$, then there is one and only one $\rho \in \left[ 0,1%
\right] $ such that $L_{3}\equiv _{A}\rho L_{1}+\left( 1-\rho \right) L_{2}$.

\begin{proof}
It follows straightforwardly from lemmas \ref{FLTA} and \ref{DBL}, Dedekind
cut and small event continuity (see page 73, \cite{Sav1954}). The Fishburn's
version is more complete (see page 205, \cite{Fish1971}).
\end{proof}
\end{lemma}

\begin{theorem}
\label{FTR}For each $A\in \Sigma\backslash \left\{ \emptyset \right\} $%
, there is a real-valued function $u_{A}$ on $O$ satisfying 
\begin{equation*}
L\geqq _{A}L^{\prime }\Leftrightarrow \sum_{o:L\left( o\right)
>0}u_{A}\left( o\right) L\left( o\right) \geq \sum_{o:L^{\prime }\left(
o\right) >0}u_{A}\left( o\right) L^{\prime }\left( o\right)
\end{equation*}%
for all simple lotteries $L$ and $L^{\prime }$ on $O$. Besides, $u_{A}$ is
unique up to a positive affine transformation.

\begin{proof}
For each $A\in \Sigma\backslash \left\{ \emptyset \right\} $, as $\geqq
_{A}$ is a weak preference satisfying independence and Archimedean
properties, using the theorem 8.2, page 107, \cite{Fish1971}, the result
follows.
\end{proof}
\end{theorem}

\begin{lemma}
For each $A,B\in \Sigma\backslash \left\{ \emptyset \right\} $ such
that $B\approx A$, $u_{A}$ and $u_{B}$ are related by a positive affine
transformation.

\begin{proof}
It is straightforward from corollary \ref{ICL}.
\end{proof}
\end{lemma}

\begin{lemma}
\label{ACL}For the set of constant lotteries, $\geqq _{A}$ and $\geqq _{B}$
agree.for each $A,B\in \Sigma\backslash \left\{ \emptyset \right\} $.

\begin{proof}
It is a straightforward consequence of eventwise monotonicity and the
definition of $\geqq _{A}$ and $\geqq _{B}$.
\end{proof}
\end{lemma}

\begin{definition}[mixture space]
A mixture space is a $2$-uple $\left( X,m\right) $ with $X$ a non empty set
and $m:\left[ 0,1\right] \times X^{2}\rightarrow X$ such that $m\left(
1,x,y\right) =x$, $m\left( \alpha ,x,y\right) =m\left( 1-\alpha ,y,x\right) $
and $m\left( \alpha ,m\left( \beta ,x,y\right) ,y\right) =m\left( \alpha
\beta ,x,y\right) $.
\end{definition}

\begin{theorem}
\label{GTRMS}For a mixture space $\left( X,m\right) $ and a weak preference $%
\succsim $ on $X$, there exist a unique, up to affine transformations,
real-valued function $v:X\rightarrow 
%TCIMACRO{\U{211d} }%
%BeginExpansion
\mathbb{R}
%EndExpansion
$ such that 
\begin{equation*}
\forall x,y\in X\left( x\succsim y\Leftrightarrow v\left( x\right) \geq
v\left( y\right) \right)
\end{equation*}%
and%
\begin{equation*}
\forall x,y\in X\forall \alpha \in \left[ 0,1\right] \left( v\left( m\left(
\alpha ,x,y\right) \right) =\alpha v\left( x\right) +\left( 1-\alpha \right)
v\left( y\right) \right) ,
\end{equation*}%
if, and only if,%
\begin{equation}
\forall x,y,z\in X\forall \alpha \in \left( 0,1\right) \left( x\succ
y\Rightarrow m\left( \alpha ,x,z\right) \succ m\left( \alpha ,y,z\right)
\right)  \label{IMS}
\end{equation}%
and%
\begin{equation}
\forall x,y,z\in X\left( x\succ y\succ z\Rightarrow \exists \alpha ,\beta
\in \left( 0,1\right) \left( m\left( \alpha ,x,z\right) \succ y\succ m\left(
\beta ,x,z\right) \right) \right) .  \label{AMS}
\end{equation}

\begin{proof}
See theorem 8.4, page 112, \cite{Fish1971}.
\end{proof}
\end{theorem}

\begin{lemma}
\label{FALC}If $O$ is convex then the set of constant acts with $m\left(
\alpha ,o,o^{\prime }\right)$ $=\alpha o+\left( 1-\alpha \right) o^{\prime
}\in O$ is a mixture space. Besides, if $\succsim _{S}$ on this mixture
space satisfies \ref{IMS} and \ref{AMS}, then $\geqq _{A}$ and $\geqq _{B}$
agree.for each $A,B\in \Sigma\backslash \left\{ \emptyset \right\} $.

\begin{proof}
A convex combination of two constant acts $o$ and $o^{\prime }$ is a
constant act $\alpha o+\left( 1-\alpha \right) o^{\prime }$ for some $\alpha
\in \left[ 0,1\right] $, so, the first part follows straightforwardly.

If $\succsim _{S}$ on this mixture space satisfies \ref{IMS} and \ref{AMS},
then, by theorem \ref{GTRMS}, there exist a unique, up to affine
transformations, real function $v:O\rightarrow 
%TCIMACRO{\U{211d} }%
%BeginExpansion
\mathbb{R}
%EndExpansion
$ such that 
\begin{equation*}
\forall o,o^{\prime }\in O\left( o\succsim _{S}o^{\prime }\Leftrightarrow
v\left( o\right) \geq v\left( o^{\prime }\right) \right)
\end{equation*}%
and%
\begin{equation*}
\forall o,o^{\prime }\in O\forall \alpha \in \left[ 0,1\right] \left(
v\left( \alpha o+\left( 1-\alpha \right) o^{\prime }\right) =\alpha v\left(
o\right) +\left( 1-\alpha \right) v\left( o^{\prime }\right) \right) .
\end{equation*}

By lemma \ref{ACL} and theorem \ref{FTR}, for each $A\in \Sigma%
\backslash \left\{ \emptyset \right\} $, $u_{A}$ and $v$ are related by a
positive affine transformation. Thus, the second part follows.
\end{proof}
\end{lemma}

\begin{corollary}
\label{CFALC}Under the conditions of lemma \ref{FALC}, there is a
real-valued function $v$ on $O$ such that, for each $A\in \Sigma%
\backslash \left\{ \emptyset \right\} $, 
\begin{equation*}
L\geqq _{A}L^{\prime }\Leftrightarrow \sum_{o:L\left( o\right) >0}v\left(
o\right) L\left( o\right) \geq \sum_{o:L^{\prime }\left( o\right) >0}v\left(
o\right) L^{\prime }\left( o\right)
\end{equation*}%
for all simple lotteries $L$ and $L^{\prime }$ on $O$. Besides, $v$ is
unique up to a positive affine transformation.

\begin{proof}
It is straightforward from lemma \ref{FALC}.
\end{proof}
\end{corollary}

It is shown above that, under the conditions of lemma \ref{FALC}, there
exists a unique, up to affine transformations, event-independent weak
preference for simple lotteries on $O$. I.e., an agent has a unique well
defined ordering on simple lotteries and discrepancies between different
classes are implied by different hypothetical conditional beliefs only. This
idea is formalized below.

\begin{definition}[$\geqq $]
Under the conditions of lemma \ref{FALC}, for each pair of simple lotteries $%
L$ and $L^{\prime }$ on $O$, 
\begin{equation*}
L\geqq L^{\prime }\Leftrightarrow L\geqq _{A}L^{\prime },
\end{equation*}%
for some $A\in \Sigma\backslash \left\{ \emptyset \right\} $.
\end{definition}

\begin{lemma}
Under the conditions of lemma \ref{FALC}, for each pair of simple acts $f$
and $g$, and each $A\in \Sigma\backslash \left\{ \emptyset \right\} $, 
\begin{equation*}
\begin{array}{c}
f\succsim _{A}g \\ 
\Leftrightarrow \\ 
L_{A}^{f}\geqq L_{A}^{g} \\ 
\Leftrightarrow \\ 
\sum_{o:L_{A}^{f}\left( o\right) >0}v\left( o\right) L_{A}^{f}\left(
o\right) \geq \sum_{o:L_{A}^{g}\left( o\right) >0}v\left( o\right)
L_{A}^{g}\left( o\right) \\ 
\Leftrightarrow \\ 
\int_{S}v\left( f\left( s\right) \right) dP_{A}\left( s\right) \geq
\int_{S}v\left( g\left( s\right) \right) dP_{A}\left( s\right)%
\end{array}%
.
\end{equation*}

\begin{proof}
By definition of $\geqq $ and corollary \ref{CFALC}.
\end{proof}
\end{lemma}

Anscombe and Aumann \cite{AnsAum1963} assumes properties similar to lemma \ref{FALC}, but in
this work there is not an explicit mixture space structure on non constant
acts (i.e. Savage's axioms are needed) and those properties on constant acts
are assumed because they imply the uniqueness mentioned above.

\end{document}